\documentclass[11pt]{article}
\usepackage{amsmath}
\usepackage{amsfonts}
\usepackage{latexsym}
\usepackage{amsthm}
\usepackage{amssymb}
\usepackage{enumerate}
\usepackage{hyperref}

\usepackage{mathtools}
\mathtoolsset{showonlyrefs}

\newcommand{\al}[1]{
  \begin{align}
  #1
  \end{align}
}

\newcommand{\E}[1]{\mathbb{E} \left[#1\right]}
\newcommand{\fl}[1]{\left\lfloor #1\right\rfloor}
\newcommand{\cl}[1]{\left\lceil #1\right\rceil}

\def\dist{\mathrm{dist}}
\def\CD{\mathrm{CD}}

\def\bp{{\bf p}}
\def\bd{{\bf d}}
\def\bi{{\bf i}}

\def\AA{\mathcal{A}}
\def\BB{\mathcal{B}}
\def\CC{\mathcal{C}}
\def\EE{\mathcal{E}}
\def\FF{\mathcal{F}}
\def\GG{\mathcal{G}}
\def\HH{\mathcal{H}}
\def\NN{\mathcal{N}}
\def\PP{\mathcal{P}}
\def\QQ{\mathcal{Q}}
\def\SS{\mathcal{S}}
\def\TT{\mathcal{T}}

\def\ol{\overline}

\def\F{\Phi}
\def\a{\alpha}
\def\b{\beta}
\def\d{\delta}
\def\D{\Delta}
\def\e{\varepsilon}
\def\f{\phi}
\def\vf{\varphi}
\def\g{\gamma}
\def\G{\Gamma}
\def\k{\kappa}
\def\th{\theta}
\def\lam{\lambda}
\def\m{\mu}
\def\r{\rho}
\def\s{\sigma}
\def\t{\tau}
\def\om{\omega}
\def\OM{\Omega}
\newcommand\Prob[1]{{\mbox{Pr}\left\{#1\right\}}}
\newcommand{\bfrac}[2]{\left({\frac{#1}{#2}}\right)}

\newtheorem{theorem}{Theorem}
\numberwithin{theorem}{section}
\newtheorem{proposition}[theorem]{Proposition}
\newtheorem{lemma}[theorem]{Lemma}
\newtheorem*{lemma*}{Lemma}
\newtheorem{corollary}[theorem]{Corollary}
\newtheorem{definition}[theorem]{Definition}
\newtheorem{claim}{Claim}
\numberwithin{claim}{section}
\newtheorem{subclaim}{Subclaim}
\newtheorem*{open}{Open Problem}
\newtheorem{conjecture}{Conjecture}

\title{A condition for Hamiltonicity in Sparse Random Graphs with a Fixed Degree Sequence}

\author{Tony Johansson\thanks{Funded by the Swedish Research Council (grant 2015-05015).} \\[.2cm]
  {\small Stockholm University} \\
  {\small Stockholm, Sweden}
}

\begin{document}
\maketitle

\begin{abstract}
We consider the random graph $G_{n, \bd}$ chosen uniformly at random from the set of all graphs with a given sparse degree sequence $\bd$. We assume $\bd$ has minimum degree at least 4, at most a power law tail, and place one more condition on its tail. For $k\ge 2$ define $\b_k(G) = \max e(A, B) + k(|A|-|B|) - d(A)$, with the maximum taken over disjoint vertex sets $A, B$. It is shown that the problem of determining if $G_{n, \bd}$ contains a Hamilton cycle reduces to calculating $\b_2(G_{n, \bd})$. If $k\ge 2$ and $\d\ge k+2$, the problem of determining if $G_{n, \bd}$ contains a $k$-factor reduces to calculating $\b_k(G_{n, \bd})$.
\end{abstract}

\section{Introduction}

We consider the Fixed Degree Sequence random graph $G_{n, \bd}$, defined as follows. Suppose $\bd_n = (d_1^{(n)} \ge \dots \ge d_n^{(n)})$ is a sequence of non-negative integers, and let $\bd = (\bd_n)$ be a sequence of degree sequences. We let $\GG_{n, \bd}$ be the set of (simple) graphs on $V_n = \{v_1,\dots,v_n\}$ where $d(v_i) = d_i^{(n)}$ for all $i$. Let $G_{n, \bd}$ be the graph chosen uniformly at random from $\GG_{n, \bd}$.

Our main concern in this paper is the problem of determining whether or not $G_{n, \bd}$ contains a Hamilton cycle, i.e. a cycle passing through each vertex exactly once. Frieze's recent survey \cite{Frieze19} is an excellent introduction to the problem. Finding Hamilton cycles in random graphs has been an active research area since Erd\H{o}s and R\'enyi first posed the question in 1960 \cite{ErdosRenyi60} for the graph $G_{n, m}$ chosen uniformly at random from graphs on $n$ vertices and $m$ edges. This was later solved by Kor\v{s}unov \cite{Korsunov} and made more precise by Koml\'os and Szemer\'edi \cite{KomlosSzemeredi}, with many extensions in subsequent years. The Erd\H{o}s-R\'enyi graph $G_{n, m}$ requires $m = \OM(n\ln n)$ edges to be Hamiltonian, essentially because this is the number of edges needed for the minimum degree to be at least 2.


Hamiltonicity has been studied and solved in other random graph models. The $d$-out graph is given by each vertex $v\in V$ choosing a set of $d$ vertices $A(v)\subseteq V$ uniformly  at random, and letting $G_{d-out}$ be the graph where the edge $uv$ is included if and only if $u\in A(v)$ or $v\in A(u)$. Whp\footnote{A sequence of events $\EE_n$ happens {\em with high probability (whp)} if $\Prob{\EE_n}\to 1$.}, $G_{d-out}$ is Hamiltonian if and only if $d\ge 3$ \cite{BohmanFrieze09}. The random $d$-regular graph is a special case of $G_{n, \bd}$ where $d_i^{(n)} = d$ for all $i, n$. This is known to be Hamiltonian whp if and only if $d\ge 3$ \cite{RobinsonWormald,CooperFriezeReed,KrivelevichSudakovVuWormald}. Note that for fixed $d$ these models are sparse, i.e. contain $O(n)$ edges. Notable unsolved models are the Preferential attachment model and related Uniform attachment model, where it is known that Hamiltonicity holds whp for $\d\ge D$ for some large constant $D$ \cite{FPPR}.

As the examples above suggest, the minimum degree of a random graph is one of the main obstructions to its Hamiltonicity. Motivated by this one may study $G_{n, m}$ conditional on having minimum degree $d$. Define a model $G_{n, m}^{(d)}$ by
$$
\Prob{G_{n, m}^{(d)} = G} = \Prob{G_{n, m} = G \mid \d(G_{n, m})\ge d},\quad \d(G) \ge d.
$$
Following work by Bollob\'as, Cooper, Fenner and Frieze \cite{BollobasCooperFennerFrieze}, Frieze \cite{Frieze14} and most recently by Anastos and Frieze \cite{AnastosFrieze19}, it is known among other things that $G_{n, cn}^{(3)}$ is Hamiltonian whp if $c > 2.66\dots$. In his survey, Frieze asks if $G_{n, cn}^{(d)}$ contains a Hamilton cycle whenever $c > d/2$ \cite[Problem 4]{Frieze19}. The current paper addresses this problem, as $G_{n, cn}^{(d)}$, conditional on its degree sequence, is a special case of $G_{n, \bd}$.

For $d\ge 1$ the {\em $d$-core} of a graph $G$, denoted $C_d(G)$ is the largest possible induced subgraph of $G$ with minimum  degree at least $d$. The problem of determining if $C_d(G_{n, m})$ is Hamiltonian essentially reduces to studying $G_{n, cn}^{(d)}$. Krivelevich, Lubetzky and Sudakov \cite{KrivelevichLubetzkySudakov} showed that the $d$-core is Hamiltonian whp when $d\ge 15$, and as the authors remark it is believed that $d\ge 3$ is enough.

It should be noted that many recent results on Hamiltonicity, notably \cite{BollobasCooperFennerFrieze} and \cite{KrivelevichLubetzkySudakov} among the ones mentioned, extend to the existence of several edge-disjoint Hamilton cycles. The current paper is restricted to a single Hamilton cycle.

Circling back to the model of interest, namely $G_{n, \bd}$, the best known result so far for general (non-regular) $\bd$ is given by Cooper, Frieze and Krivelevich \cite{CooperFriezeKrivelevich}. Under relatively complicated conditions on $\bd$, in particular assuming that the average degree is large enough, they show that $G_{n, \bd}$ is Hamiltonian whp. Gao, Isaev and McKay \cite{GaoIsaevMckay} prove Hamiltonicity when $\bd$ is near $d$-regular for some $d\gg \ln n$. Frieze's survey asks  for a simple function $\f$ such that $G_{n, \bd}$  is Hamiltonian whp if and only if $\f(\bd) > 0$ \cite[Problem 24]{Frieze19}. We are not able to fully answer this question, but reduce it to the open problem below.

As noted in \cite{CooperFriezeKrivelevich}, a minimum degree condition on $\bd$ is not enough to guarantee Hamiltonicity in $G_{n, \bd}$, even with bounded maximum degree. Indeed, suppose we take $\bd$ to have a set $A$ of $2n/3$ vertices of degree $\d$ and a set $B$ of $n/3$ vertices of degree $D$ for some large $D$. Then as $D$ grows, the expected number of edges in $A$ is $\e_Dn$ for some $\e_D\to 0$. As a Hamilton cycle must have at least $n/3$ edges contained in $A$, there exists some finite $D$ such that $G_{n, \bd}$ is non-Hamiltonian with some positive probability.

In order for $G_{n, \bd}$ to be Hamiltonian, some restriction needs to be put on the shape of $\bd$, ensuring that edge densities are not too skewed as in the above example. For disjoint vertex sets $A, B$ and $k\ge 2$ we define
$$
\b_k(A, B) = \max_{A\cap B = \emptyset} e(A, B) + k(|A| - |B|) - d(A),
$$
and define $\b_k(G)$ as the maximum value of $\b_k(A, B)$ over all disjoint sets $A, B$. We say that $G$ is {\em ($k$-)balanced} if $\b_k(G) \le 0$ and {\em ($k$-)unbalanced} if $\b_k(G) > 0$. It is easy to see (Proposition~\ref{prop:window}) that 2-unbalanced graphs are not Hamiltonian. Our main result states that if $\bd$ has minimum degree at least 4 and is such that $G_{n, \bd}$ is {\em strongly 2-balanced}, meaning $\b_2(A, B) \le -\th |A|$ for all $A, B$ for some $\th > 0$, then $G_{n, \bd}$  is Hamiltonian whp, with some additional mild conditions on $\bd$. 
\begin{open}\label{open:balance}
For $k\ge 2$, determine for which degree sequences $\bd$ the graph $G_{n, \bd}$ is strongly $k$-balanced whp.
\end{open}

\subsection{Main result}\label{sec:mainresult}

Before stating the main result we need to discuss degree sequences. For a degree sequence $\bd_n = (d_1^{(n)} \ge \dots \ge d_n^{(n)})$, define its associated degree profile $\bp_n = \bp(\bd_n) = (p_d^{(n)} : d \ge 0)$ by letting $p_d^{(n)}$ equal the number of $i$ for which $d_i^{(n)} = d$.

We will generally refer to the sequence of degrees sequences $\bd = (\bd_n)$ itself as a degree sequence. This is consistent with how random graph sequences are commonly called random graphs, and should cause no confusion.

The minimum degree $\d$ of $\bd$ is defined as
$$
\d = \inf_n d_n^{(n)},
$$
while the maximum degree is defined as a function of $n$, namely $\D_n = d_1^{(n)}$. If $\sup_n \D_n < \infty$ we say that $\bd$ is {\em bounded}, otherwise it is {\em unbounded}. Define for $s \in \{1,\dots,n\}$,
$$
D(s) = \max_{|S| = s} \sum_{v\in S} d(v).
$$
\begin{definition}\label{def:seqprops}
Suppose $\bd$ is a degree sequence with associated profile $\bp$.
\begin{enumerate}[(a)]
\item The total number of edges of $\bd$ is $m(\bd) = \frac12\sum_{i=1}^n d_i^{(n)}$. The sequence $\bd$ is {\em sparse} if $\limsup_{n} n^{-1}m(\bd) < \infty$.
\item The sequence $\bd$ is {\em linearly unbounded} if for every integer $D\ge \d$ there exists some $\e > 0$ such that
$$
\liminf_{n\to\infty} \frac1n \sum_{d \ge D} p_d^{(n)} \ge \e.
$$
\item The sequence $\bd$ is {\em light-tailed} if $d_1^{(n)} \le n^{1/6}$ and there exist constants $C > 0, 1/2 < \a < 1$ such that for all $t$ and $n$,
\begin{equation}\label{eq:lighttail}
D(t) \le Cn\bfrac{t}{n}^{\a}
\end{equation}
\item\label{def:degbal} Let $k\ge 1$. The sequence $\bd$ is {\em strongly $k$-balanced} if $G_{n, \bd}$ is strongly $k$-balanced with high probability.
\end{enumerate}
\end{definition}
\noindent Any of these properties held by $\bd$ is said to be held by its associated profile $\bp$ as well. Condition \eqref{eq:lighttail} can be less formally stated as the degree sequence, at worst, following a power law $p_d \sim d^{-\a}$ with $\a > 2$. We can now state our main theorem. 
\begin{theorem}\label{thm:main}
Suppose $\bd$ has minimum degree $\d\ge 4$ and is linearly unbounded, sparse and light-tailed. 
\begin{enumerate}[(i)]
\item If $\bd$ is strongly 2-balanced then $G_{n, \bd}$ is Hamiltonian whp,
\item Let $2\le k\le \d-2$. If $\bd$ is strongly $k$-balanced, then $G_{n, \bd}$ contains a $k$-factor whp.
\end{enumerate}
\end{theorem}
Theorem~\ref{thm:main} is unsatisfactory in more than one way. Most notably it says nothing about which degree sequences are strongly balanced. We would also like to reduce the minimum degree condition to $\d\ge 3$, and generally expect $G_{n, \bd}$ to contain a $(\d-1)$-factor. Define $\AA_k$ as the property of containing $\fl{(k-1)/2}$ edge-disjoint Hamilton cycles and, if $k$ is odd, one perfect matching disjoint from the Hamilton cycles. We make a conjecture.
\begin{conjecture}\label{conj}
Let $k\ge 2$ and $\d \ge k+1$. Suppose $\bd$ is a sparse, light-tailed and strongly $k$-balanced degree sequence with minimum degree $\d$. Then $G_{n, \bd}$ has property $\AA_{k}$ whp.
\end{conjecture}
This conjecture does not assume that $\bd$ is linearly unbounded. In Section~\ref{sec:rerandom} we discuss the limitations of the proof technique used here, which led us to work with linearly unbounded $\bd$.

In conjunction with a satisfactory solution to Open Problem~\ref{open:balance}, proving this conjecture would largely settle the question of Hamiltonicity in fixed degree sequence random graphs. We discuss steps toward extending the present result to Conjecture~\ref{conj} in Section~\ref{sec:concluding}.

\subsection{Proof outline}

On a certain level the proof method should be familiar to readers experienced with Hamiltonicity results in random graphs, and we give a brief outline of the proof here. This outline concerns Hamiltonicity, which is the main focus of the paper, and results on factors can essentially be viewed as a by-product of this proof.

The result will be shown for the fixed-profile graph $G_{n, \bp}$, chosen uniformly at random from the set of graphs with degree profile $\bp$. This can be seen as $G_{n, \bd}$ with the vertices randomly permuted, and this does not change the probability of containing a Hamilton cycle.

The reason for working with $G_{n,\bp}$ is that it can be rerandomized, as we will show. This involves defining a random subgraph $\G\subseteq G = G_{n, \bp}$, obtained by removing $r$ random edges and designed to maintain the same minimum degree as $G$. This is defined in such a way that $G$ is reobtained by adding random edges to $\G$, and we show that the first $s$ such edges are nearly uniform in distribution for some $s < r$.

For a graph $G$ let $\lam_2(G)$ be the minimum value of $n-|F|$ taken over subgraphs $F\subseteq G$ with maximum degree at most 2, so that $\lam_2(G) = 0$ if and only if $G$ contains a 2-factor. We show that $\G$ has a number of properties, defined in terms of certain edge densities and cycle lengths. One of these properties (balance) will directly imply that $\lam_2(\G) = o(s)$. The other properties ensure that there is a set of $\OM(n^2)$ edges $e\notin \G$ such that $\lam_2(\G + e) < \lam_2(\G)$. The first $s/2 \gg \lam_2(\G)$ near-uniform edges we add will with high probability give a graph $\G'$ with $\lam_2(\G') = 0$.

The graph $\G'$ contains a 2-factor, and our graph properties will imply that the number of cycle components is $o(s)$. We argue that there are $\OM(n^2)$ edges $e\notin \G'$ whose addition reduces the optimal component number among 2-factors in $\G'$, and add the remaining $s/2$ edges to obtain a Hamiltonian graph $\G''$. Adding the remaining $r-s$ edges gives a graph which is Hamiltonian and is close to $G_{n, \bp}$ in distribution, and we conclude that $G_{n, \bp}$ is Hamiltonian whp.

A key theme of the proof is defining properties which are either monotone or have good tolerance to edges being added and removed. In conjunction with the minimum degree being maintained, this allows us to prove the existence of the $\OM(n^2)$ edges whose addition improves our graphs. In particular we do not use P\'osa's lemma, which is the traditional tool for accomplishing this.

The paper follows a different path than this outline, and we frequently postpone proofs to the bottom of the paper in order to highlight the high-level proof before going into details.

\section{Sparsity, $k$-matchings, girth and balance}\label{sec:properties}

Some graph concepts need to be defined before we state any results, namely {\em sparsity}, {\em $k$-matchings}, {\em effective girth} and {\em balance}.

Suppose $A, B$ are vertex sets and $v$ a vertex. We let $d(v)$ denote the degree of a vertex, and $d(A) = \sum_{v\in A} d(v)$. The number of edges in $A$ is denoted by $e(A)$, and the number of edges between $A$ and $B$ by $e(A, B)$.

Firstly, for $c, \g > 0$ say that a graph on $n$ vertices is $(c, \g)$-sparse if $e(S) \le c|S|$ for every vertex set $S$ of size at most $\g n$. Define a graph class
\begin{equation}\label{eq:SSdef}
\SS_\g = \{G : \text{ $G$ is $(13/12, \g)$-sparse}\}.
\end{equation}

\subsection{$k$-matchings}

Let $k\ge 1$ be an integer. An edge set $F$ is called a {\em $k$-matching} if $d_F(v) \le k$ for all vertices $v$. We define a {\em $k$-factor} as a $k$-matching with $|F| = kn / 2$, noting that one can exist only if $kn$ is even. Statements about $G_{n, \bp}$ containing a $k$-factor for odd $k$ should be read as holding when restricted to the subsequence of even $n$. We define
\al{
  \lam_k(F) & = \frac{kn}{2} - |F|,
}
and $\lam_k(G)$ as the minimum of $\lam_k(F)$ over $k$-matchings $F\subseteq G$.

\subsection{Effective girth}\label{sec:girth}

\begin{definition}\label{def:girth}
For a vertex $v\in G$ let $c_G(v)$ denote the length of the shortest cycle in $G$ containing $v$. For $k\ge 3$ let $c_k(G)$ denote the number of vertices $v$ with $c_G(v)\le k$. The {\em effective girth} of $G$, denoted $\psi(G)$, is the largest integer $\psi$ for which
$$
\psi c_\psi \le n.
$$
\end{definition}
Note that for any $2$-factor $F$ in a graph $G$, the number of connected components $\k(F)$ satisfies
\begin{equation}\label{eq:girthbound}
\k(F) \le \frac{2n}{\psi(G)}. 
\end{equation}
Indeed, let $F$ be a $2$-factor, and for each vertex $v$ let $C_F(v)$ denote the number of vertices in the connected component of $v$ in $F$. Then $C_F(v) \ge c_G(v)$, so
\al{
  \k(F) = \sum_v \frac{1}{C_F(v)} & \le \sum_v \frac{1}{c_G(v)} \\
  & \le \sum_{v : c(v) \le \psi(G)} \frac{1}{c(v)} + \sum_{v : c(v) > \psi(G)} \frac{1}{\psi(G)} \\
  & \le c_\psi + \frac{n}{\psi(G)} \le 2\times \frac{n}{\psi(G)}.
}
Define the graph class
$$
\CC = \{G : \text{ $G$ has effective girth at least $\ln^{1/2}n$}\}.
$$
\subsection{Balance}\label{sec:balance2}

For disjoint vertex sets $A, B$ and integers $k\ge 1$ define
$$
\beta_k(A, B) = e(A, B) + k(|A| - |B|) - d(A).
$$
We define $\beta_k(G) = \max \beta_k(A, B)$, with the maximum taken over disjoint vertex sets $A, B$. The following proposition shows how $\beta_k$ is related to $k$-factors. We remark that the term $n\psi^{-1/2}$ has not been optimized, and for large $\psi$ it can be replaced by a term of order $n\frac{\ln\ln \psi}{\ln\psi}$.
\begin{proposition}\label{prop:window}
Let $k\ge 1$. Suppose $G$ is a sparse graph on $n$ vertices and $m = \m n$ edges. There exists a constant $C > 0$ such that if $G$ has effective girth $\psi \ge C\m^2$ then
$$
\beta_k(G) \le 2\lam_k(G) \le \frac{n}{\psi^{1/2}} + \max\left\{0, \beta_k(G)\right\},
$$
\end{proposition}
\begin{proof}
The upper bound is restated and proved in Section~\ref{sec:bipartite}. We show the lower bound here. Suppose $F$ is a $k$-matching with $\lam_k(F) = \lam_k(G)$, and suppose $A, B$ are sets with $\beta_k(A, B) = \beta_k(G) > 0$. Also let $C = V\setminus (A\cup B)$. Then
\begin{multline}
d_F(A) = e_F(A, B) + e_F(A, C) + 2e_F(A) \\
\le k|B| + e(A, C) + 2e(A) \\
   \le k|B| + d(A) - e(A, B)  = k|A| - \beta_k(A, B).
\end{multline}
This implies
\al{
  2\lam_k(G) = \sum_{v\in V} (k - d_F(v)) & \ge \sum_{v\in A} (k-d_F(v)) \\
  & = k|A| - d_F(A) \ge  \beta_k(A, B).
}
\end{proof}
For $k\ge 2$ define the graph property
$$
\BB^{(k)} = \{G : \beta_k(G) \le 0\},
$$
and for $\t, \th > 0$ define
$$
\BB^{(k)}_{\t, \th} = \left\{G \in \BB^{(k)} : \max_{\substack{A\cap B = \emptyset \\ |A| \ge \t n}} \beta_k(A, B) \le -\th n\right\}.
$$
While $\SS_\g, \CC$ are decreasing properties, for $\BB^{(k)}_{\t,\th}$ we  will make use of the following lemma.
\begin{lemma}\label{lem:StoB}
Let $k\ge 2$ and let $\th, \g, \vf > 0$ with $2\vf < \th$. Suppose $G\in \BB^{(k)}_{\g/2, \th}\cap \SS_\g$ has minimum degree $\d\ge  k+2$, and suppose $H\subseteq G$ has $e(H) \ge e(G) - \vf n$. Then $H\in \BB^{(k)}$.
\end{lemma}
\begin{proof}
Let $A, B$ be disjoint vertex sets. First suppose $|A\cup B| \le \g n$. Then, since $\SS_\g$ is a decreasing property we have
\al{
  \beta_k^H(A, B) & = e_H(A, B) + k(|A| - |B|) - d_H(A) \\
  & \le e_H(A\cup B) + k(|A| - |B|) - \d|A| \\
  & \le \frac{13}{12}|A\cup B| - 2|A| - 2|B| \le 0.
}
Now suppose $|A\cup B| > \g n$. If $|A| \le |B|$ then $\beta_k(A,B)\le 0$ clearly holds, so suppose $|A| \ge \g n / 2$. Then
\begin{multline}
  \beta_k^H(A, B) = e_H(A, B) + k(|A| - |B|) - d_H(A) \\
  \le e_G(A, B) + k(|A| - |B|) - d_G(A) + 2\vf n 
  \le -\th n + 2\vf n.
\end{multline}
\end{proof}

\section{$k$-matchings and targets}\label{sec:altwalk}

For a $k$-matching $F$ define
$$
X_F = \{v\in V : d_F(v) < k\}.
$$
As in Section~\ref{sec:properties}, for $k$-matchings $F$ define
$$
\lam_k(F) = \fl{\frac{kn}{2}} - |F|, \quad \lam_k(G) = \min_{F\subseteq G} \lam_k(F).
$$
Let $\FF_k(G)$ denote the family of $k$-matchings $F\subseteq G$ with $\lam_k(F) = \lam_k(G)$. Let
$$
X = \bigcup_{F\in \FF_k(G)} X_F,
$$
and for $x\in X$ let
$$
Y_x = \bigcup_{\substack{F\in \FF_k(G) \\ x\in X_F}} (X_F\setminus \{x\}).
$$
Let $Y_x=\emptyset$ for $x\notin X$. Note that if $x\in X$ and $y\in Y_x$ then $\lam_k(G + xy) < \lam_k(G)$.

A $2$-matching $F$ of size $n$ is a $2$-factor, or in other words a partition of $V$ into disjoint cycles. For an edge set $F$ let $\k(F)$ denote the number of connected components of $F$, and note that a Hamilton cycle is a 2-factor with $\k = 1$. Let $\k_2(G)$ be the minimum value of $\k(F)$ over 2-factors $F$ in $G$. If $\lam_2(G) = 0$ and $\k_2(G) > 1$, for each $w\in V$ let
$$
Z_w = \{z\in V : \k_2(G + wz) < \k_2(G)\}.
$$
Let $W$ be the set of vertices $w$ such that $Z_w$ is nonempty.

\begin{lemma}\label{lem:posaish}
Suppose $G$ is a graph on $n$ vertices.
\begin{enumerate}[(i)]
\item Let $k\ge 2$ and suppose $\d(G) \ge k+1$, $\lam_k(G)  > 0$ and let $x\in X$. Then there exists a nonempty set $S_x$ such that
$$
e(S_x) \ge \frac{13}{12}|S_x|,
$$
and $|S_x\cap  Y_x| \ge \frac{1}{k+1}|S_x| - 1$.
\item Suppose $\d(G)\ge 4$, $\lam_2(G) = 0$, $\k_2(G) > 1$. Then for every vertex $w\in W$, there exists a nonempty $T_w$ such that
$$
e(T_w)  \ge \frac{7}{6}|T_w|,
$$
and $|T_w\cap Z_w| \ge |T_w|/3$.
\end{enumerate}
\end{lemma}
We postpone the proof of Lemma~\ref{lem:posaish} to Section~\ref{sec:posaish}.

We introduce the concept of {\em targets}. A target for a graph $G$ is an edge $e\notin G$ such that $G+e$ is closer to containing a structure of interest. In the context of $k$-factors we define $\TT_k(G)$ as the set of $e\notin G$ for which $\lam_k(G + e) < \lam_k(G)$, noting that $G$ contains a $k$-factor if and only if $\TT_k(G) = \emptyset$. In the context of Hamilton cycles, the definition of the target set $\TT_c(G)$ has three possible definitions:
$$
\begin{array}{ll}
\{xy\notin G : \lam_2(G + e) < \lam_2(G)\}, & \lam_2(G) > 0, \\
\{xy\notin G : \k(G + e) < \k(G)\}, & \lam_2(G) = 0, \k(G) > 1, \\
\{xy\notin G : \k_2(G + e) < \k_2(G)\}, & \lam_2(G) = 0, \k(G) = 1, \k_2(G) > 1.
\end{array}
$$
When $\lam_2(G) = 0, \k(G) = 1$ and $\k_2(G) = 1$ then $G$ is Hamiltonian, and we let $\TT_c(G) = \emptyset$. 
\begin{corollary}\label{cor:target} Let $0 < \g < 1/2$. Suppose $G$ is a $(13/12,\g)$-sparse graph on $n$ vertices with effective girth at least $\ln^{1/2}n$. Let $H$ be a graph containing $G$. For $n$ large enough, the following holds.
\begin{enumerate}[(i)]
\item If $k\ge 2$ and $\d(G) \ge k+2$ then at least one of the following holds: (a) $H$ contains a $k$-factor, (b) $|\TT_k(H)| \ge \frac{\g^2}{2(k+2)^2}n^2$, (c) $H$ is not $(13/12,\g)$-sparse.
\item If $\d(G) \ge 4$ then at least one of the following holds: (a) $H$ is Hamiltonian, (b) $|\TT_c(H)| \ge \g^2n^2/32$, (c) $H$ is not $(13/12, \g)$-sparse.
\end{enumerate}
\end{corollary}
\begin{proof}[Proof of Corollary \ref{cor:target}]
We prove (i), so let $k\ge 2$ and assume $\d(G) \ge k+1$ and $G\subseteq H$. Suppose $\lam_k(H) = 0$ and that $H$ is $(13/12, \g)$-sparse. We prove that $|\TT_k(H)| = \OM(n^2)$. 

Let $x\in X(H)$. Let $S_x$ be as in Lemma~\ref{lem:posaish} (i). As $H$ is $(13/12,\g)$-sparse, we have $|S_x| \ge \g n$, and $|Y_x| \ge |S_x\cap Y_x| \ge \frac{1}{k+1}|S_x|-1 \ge \frac{\g}{k+2} n$ for $n$ large enough. Since $Y_x\subseteq X$, we have $|X| \ge  \frac{\g}{k+2} n$, and
$$
|\TT_k(H)| = \frac12 \sum_{x, y} [x\in X, y\in Y_x] \ge \frac12 \sum_{x\in X} |Y_x| \ge \frac{\g^2}{2(k+2)^2}n^2.
$$
We prove (ii). The case $\lam_2(H) > 0$ is exactly as (i). Suppose $\lam_2(H) = 0$ and $\k(H) > 1$, and that $H$ is $(13/12, \g)$-sparse. Suppose $C$ is a connected component of $H$ with $|C| \le \g n$. Then
$$
0 = e(C, \ol{C}) \ge d(C) - 2e(C) \ge 4|C| - \frac{26}{12} |C| > 0,
$$
a contradiction. So all connected components have size at least $\g n$. Since $C\times \ol C\subseteq \TT_c(H)$, we have $|\TT_c(H)| \ge |C|(n-|C|) \ge \g(1-\g)n^2 \ge \frac{\g^2}{9}n^2$.

Still assuming $\lam_2(H) = 0$ and $H\in \SS_\g$, suppose $H$ is connected and let $F$ be a 2-factor with $\k(F) = \k_2(H) > 1$. Since $H$ is connected, there exists an edge $uv\notin F$ with $u, v$ in different components of $F$. Let $u', v'$ be such that $uu', vv'\in F$. Then $\k_2(H + u'v') < \k_2(H)$, since $F' = F\triangle (u', u, v, v', u')$ is a 2-factor with $\k(F') = \k(F) - 1$. Taking $w = u'$, we conclude that $w\in W$ so $W\ne\emptyset$. We can now apply the same line of reasoning as in the previous paragraph, applying Lemma~\ref{lem:posaish}~(ii) to conclude that $|Z_w| \ge \g n / 3$. But $Z_w\subseteq W$, so
$$
|\TT_c(H)| = \frac12 \sum_{w, z} [w\in W, z\in Z_w] \ge  \frac12 \sum_{w\in W}|Z_w| \ge \frac{\g^2}{9}n^2. 
$$
\end{proof}

\section{Sprinkling}\label{sec:sprinkling}

\begin{definition}\label{def:random}
Let $E$ be a finite set and $\t,\e\in [0,1]$. A probability measure $\m$ on $E$ is {\em $(\t, \e)$-random} if for any set $S\subseteq E$ with $|S| \ge \e |E|$,
$$
\m(S) \ge \t.
$$
\end{definition}
Let $G = (V, E)$ be a graph and $s\ge 0$ an integer. Let
$$
G^{(s)} = \{G\cup H : |H|\le s, H\cap E = \emptyset\}
$$
be the set of graphs obtained by adding at most $s$ edges to $G$. Suppose $\m = \{\m_\G : \G\in G^{(s-1)}\}$ is a family of measures, where each $\m_\G$ is a probability measure on $\G^{(1)}$, or equivalently a probability measure on the set of edges not in $\G$. Let $\SS^{(0)} = \{G\}$, and recursively let $\SS^{(i)}$ be the set of $\G_1\in G^{(i)}$ such that $\m_\G(\G_1) > 0$ for some $\G\in \SS^{(i-1)}$, for $i=1,\dots,s$. Say that $\m$ has {\em support} $\SS = \SS^{(0)}\cup\dots\cup \SS^{(s)}$. We call $\m = \{\m_\G : \G\in \SS\}$ an {\em $s$-round sprinkling scheme for $G$}.

Given an $s$-round sprinkling scheme for $G$, we can define a sequence $\m^t[G], t = 0,\dots,s$, by letting $\m^0[G] = G$ and successively adding an edge to $H = \m^t[G]$ drawn according to $\m_H$ to form $\m^{t+1}[G]$.
\begin{definition}\label{def:sprinkler}
Let $\t, \e > 0$ and let $s\ge 0$ be an integer. A {\em $(\t, \e, s)$-sprinkler for $G$} is an $s$-round sprinkling scheme for $G$ with support $\SS$ such that for every $\G\in \SS$, $\m_\G$ is $(\t, \e)$-random.
\end{definition}
Recall from \eqref{eq:SSdef} that we define $\SS_\g$ as the class of $(13/12, \g)$-sparse graphs, and $\CC$ the class with effective girth at least $\ln^{1/2}n$. Let $\FF_k$ denote the class of graphs containing a $k$-factor, and $\HH$ the class of graphs containing a Hamilton cycle.

\begin{proposition}\label{prop:sprinkling}
Let $k \ge 2$. Let $\t,\e,\g,\s > 0$ with $\g < 1/2$ and $\t < \frac{\g^2}{2(k+2)^2}$. Let $s = \s n$. Suppose $G \in \SS_\g\cap \CC$ is balanced. Suppose $\m$ is a $(\t,\e,s)$-sprinkler for $G$. 
\begin{enumerate}[(i)]
\item Suppose $k\ge 2$ and $\d(G) \ge k+2$.  Then
$$
\Prob{\m^s[G] \in \FF_k \cup \ol{\SS_\g}} = 1-o(1).
$$
\item Suppose $\d(G) \ge 4$. Then
$$
\Prob{\m^s[G] \in \HH\cup \ol{\SS_\g}\cup\ol{\CC}} = 1-o(1).
$$
\end{enumerate}
\end{proposition}
\begin{proof}
We prove (ii), and then show how (i) follows by the same argument. Fix some $G$ and consider the random sequence $\m^t[G]$, $t = 0,\dots,s$. Define a (possibly infinite) random variable
$$
\s = \inf\{0 \le t \le s : \m^t[G]\in \HH\cup \ol{\SS_\g}\cup \ol{\CC}\}.
$$
Let $e_1,\dots,e_s$ be the edges added to $G$, i.e. $\m^t[G] = \m^{t-1}[G]\cup \{e_t\}$ for all $t$. Recall the definition of $\TT_c(H)$ from Section~\ref{sec:altwalk}. Let $I_t$ denote the indicator variable for the event $e_t\in \TT_c(\m^{t-1}[G])$, and let $I = I_1 + \dots + I_\s$.

We argue that $I < n/\psi^{1/2} + 1/\g + 2n/\psi$ in any outcome. Let $k = 2$. First, Proposition~\ref{prop:window} and $\beta_k(G) \le 0$ imply that
$$
\lam_k(G) \le \frac{n}{\psi^{1/2}} + \max\{0, \beta_k(G)\} = \frac{n}{\psi^{1/2}}.
$$
If $\lam_k(\m^{t-1}[G]) > 0$, $I_t=1$ implies $\lam_k(\m^t[G]) < \lam_k(\m^{t-1}[G])$, so $I \ge n/\psi^{1/2}$ implies that $\m^s[G]\in \FF_k$.

Suppose for some $t < \s$ that $\m^t[G] \in \FF_2\cap  \SS_\g$. If $H$ is not connected, then it has at most $\k(H) \le 1/\g$ connected components (as was argued in the proof of Corollary~\ref{cor:target}). So if $I_{t+1}+\dots+I_\s > 1/\g$ then $\m^s[G]$ is connected and in $\FF_2$.

Now suppose for some $t < \s$ that $H = \m^t[G]\in \FF_2\cap \SS_\g \cap \CC$, and that $H$ is connected. As noted in Section~\ref{sec:girth}, specifically equation  \eqref{eq:girthbound}, we have $\k_2(H) \le 2n / \psi(H) \le 2n/\ln^{1/2}n$. We have
$$
1\le \k_2(\m^\s[G]) \le \k_2(\m^t[G]) - (I_{t+1} + \dots + I_\s) 
$$
so $I_{t+1} + \dots + I_\s < \k_2(\m^s[G]) \le 2n/\ln^{1/2}n$.

This shows that $I_1 + \dots + I_\s < n/\psi + 1/\g + 2n/\psi \le 2n/\psi^{1/2}$. By Corollary~\ref{cor:target}~(ii), for all $t < \s$ we have $|\TT_c(\m^t[G])| = \OM(n^2)$. There must then exist some constant $p$ such that
$$
\Prob{I_t = 1 \mid \m^{t-1}[G], t < \s} \ge  p.
$$
We then have
\al{
  \Prob{\s > s} & \le \Prob{\mathrm{Bin}(s, p) \le \frac{2n}{\psi^{1/2}}} = o(1). \label{eq:sigmabound}
}
This finishes the proof of (ii). For (i), we take the target set to be $\TT_k(H)$ for any $H\in \SS$. Corollary~\ref{cor:target}~(i) shows that this has size $\OM(n^2)$ as long as $\lam_k(H) > 0$. Repeating the argument above, $I \ge n/\psi^{1/2}$ implies the existence of a $k$-factor, and we can again apply \eqref{eq:sigmabound}.
\end{proof}

\subsection{Rerandomizing $G_{n, \bp}$}

The relevance of sprinkling to $G_{n, \bp}$ is the following. For graphs $G$ with minimum degree $\d$ we define a set of edges $\F(G) \subseteq G$, which is essentially the set of edges $e\in G$ such that $G-e$ has minimum degree $\d$ (the full definition of $\F(G)$ is more involved, see Section~\ref{sec:rerandom}). For $G\in \GG_{n, \bp}$ we attempt to define a random sequence of subgraphs $G = G^{0} \supseteq G^{-1} \supseteq\dots\supseteq G^{-r}$, obtaining $G^{-i-1}$ by removing from $G^{-i}$ an edge chosen uniformly at random from $\F(G^{-i})$, if possible. This fails if $\F(G^{-i})$ is empty for some $i$, in which case we let $G^{-j} = \bot_j$ for all $j > i$.

Let $i > 0$. For any $H$ such that $\Prob{G^{-i} = H} > 0$ we define
$$
\nu_H(H') = \Prob{G^{-i+1} = H'\mid G^{-i} = H}.
$$
If $H$ is a graph, i.e. $H\ne \bot$, we can view $\nu_H$ as a measure on edges $e\notin H$, identifying $\nu_H(e) = \nu_H(H + e)$. This defines an $i$-round sprinkling scheme $\nu$ for graphs $H$, and if we define $\nu^j[\bot_i]$ by
$$
\Prob{\nu^j[\bot_i] = H'} = \Prob{G^{-i + j} = H' \mid G^{-i} = \bot_i},
$$
then
$$
\nu^j[G^{-i}] \stackrel{d}{=} G^{-i+j}, \quad j \le i.
$$
Let $\HH_r$ be the class of graphs $H$ with $\Prob{G^{-r} = H} > 0$, and let $\HH_r^* = \HH_r\cup \{\bot_r\}$. Let $\G_{n, \bp}^* = G^{-r}$. Let $\e > 0$ and define a random graph $\G_{n, \bp} \in \HH_r$ by
$$
\Prob{\G_{n, \bp} = H} = \Prob{\G_{n, \bp}^* = H \mid \G_{n, \bp}^* \ne \bot}.
$$
We will pick $r$ so that $\G_{n, \bp}^* \ne \bot$ whp. Then $\nu^r[\G_{n, \bp}]$ is close in distribution to $G_{n, \bp}$ in the sense that for any graph property $\mathbb{P}$,
\al{
  \Prob{G_{n, \bp}\notin \mathbb{P}} & = \Prob{\nu^r[\G_{n, \bp}^*] \notin \mathbb{P}} \\
  & \le \Prob{\G_{n, \bp}^*\ne \bot} \Prob{\nu^r[\G_{n, \bp}] \notin \mathbb{P}} + \Prob{\G_{n, \bp}^* = \bot} \\
  & \le o(1) + \Prob{\nu^r[\G_{n, \bp}] \notin \mathbb{P}}.
}
Given some $\F$, for integers $r\ge 0$ and constants $\a, \e, \s > 0$ we define $\GG_r = \GG_r(\t, \e, \s, \F)$ as the set of graphs $G\in \HH_r$ for which $\nu$ is a $(\t, \e, \s n)$-sprinkler.
\begin{proposition}\label{prop:rerandom}
Suppose $\bp$ is sparse and linearly unbounded. Let $\e > 0$. There exists a choice of $\F$ and a constant $\r_0 = \r_0(\bp, \F, \e)$ such that the following holds: for any $0 < \r \le \r_0$ there exist constants $0 < \s < \r$ and $\t > 0$ such that if $r = \r n$ then $\G_{n, \bp}^*\in \GG_r(\t, \e, \s, \F)$ whp.
\end{proposition}

\section{Proof of Theorem~\ref{thm:main}}\label{sec:Gnd}

Recall from Section~\ref{sec:properties} the graph classes
\al{
  \SS_\g & = \{G : \text{$G$ is $(13/12, \g)$-sparse}\}, \\
  \BB^{(k)} & = \{G:  \beta_k(G) \le 0\}, \\
  \CC & = \{G : \psi(G) \ge \ln^{1/2}n\}.
}
\begin{lemma}\label{lem:Gnd}
Suppose $\bp$ is a sparse and light-tailed. There exists a $\g > 0$ such that $G_{n, \bp}\in \SS_\g\cap \CC$ whp.
\end{lemma}
\noindent Lemma~\ref{lem:Gnd} is proved in Section~\ref{sec:Gndproof}.

Let $k\ge 2$. Fix some sparse, light-tailed, linearly unbounded and strongly $k$-connected $\bp$ with minimum degree $\d\ge k+2$, and let $\g > 0$ be as provided by Lemma~\ref{lem:Gnd}. As $\bp$ is strongly $k$-balanced there exists a $\th > 0$ such that whp, $G_{n, \bp} \in \BB_{\g/2, \th}^{(k)}$.

Let $0 < \e < \frac{\g^2}{2(k+2)^2}$ and let $\F, \r_0, \s, \t$ and $\GG_r = \GG_r(\t, \e, \s, \F)$ be as provided by Proposition~\ref{prop:rerandom}, so in particular $\G_{n, \bp}^*\in \GG_r$ whp whenever $r = \r n$ for some $\r\le \r_0$. Let $\r = \min\{\th / 2, \r_0\}$, set $r = \r n$, and define 
$$
\NN_r = \GG_r(\t, \e, \s) \cap \SS_\g \cap \CC \cap \BB^{(k)}.
$$
Then $\G_{n, \bp}^*\in \NN_r$ whp. Indeed, $\G_{n, \bp}^*$ is in $\GG_r(\a, \e, \s)$ whp by Proposition~\ref{prop:rerandom}, and $G_{n, \bp}\in \SS_\g \cap \CC\cap \BB^{(k)}_{\g/2, \th}$ whp implies that $\G_{n, \bp}^*\in \SS_\g\cap \CC\cap \BB^{(k)}$ since $\SS_\g\cap \CC$ is decreasing and by Lemma~\ref{lem:StoB}.

By Proposition~\ref{prop:rerandom} there exists a family of measures $\nu$ which is an $(\t, \e, \s n)$-sprinkler for any $H\in \NN_r$. For any graph property $\mathbb{P}$ we then have, recalling that $\nu^r[\G_{n, \bp}^*] \stackrel{d}{=} G_{n, \bp}$,
\al{
  \Prob{G_{n, \bp}\notin \mathbb{P}} & \le \Prob{\nu^r[\G_{n, \bp}^*] \notin \mathbb{P} \mid \G_{n, \bp}^* \in \NN_r} + \Prob{\G_{n, \bp}^*\notin \NN_r} \\
  & \le o(1) + \max_{G\in \NN_r} \Prob{\nu^r[G] \notin \mathbb{P}}.
}
With $\mathbb{P} = \HH\cup \ol{\SS_\g}\cup \ol{\CC}$, Proposition~\ref{prop:sprinkling}~(ii) implies that
$$
\Prob{G_{n, \bp}\notin \HH\cup \ol{\SS_\g}\cup \ol{\CC}} \le o(1).
$$
Lemma~\ref{lem:Gnd} then implies that
$$
\Prob{G_{n, \bp}\notin \HH} \le \Prob{G_{n, \bp}\notin \HH\cup \ol{\SS_\g}\cup \ol{\CC}} + \Prob{\G_{n, \bp} \notin \SS_\g\cap \CC} = o(1).
$$
Likewise, applying Proposition~\ref{prop:sprinkling}~(i) shows that $G_{n, \bp}\in \FF_k$ whp.

This finishes the proof of the main theorem. The remainder of the paper is devoted to postponed proofs.

\section{Proof of Proposition~\ref{prop:window}}\label{sec:bipartite}

Assume $G$ is a graph on $n$ vertices and $m=\m n$ edges with effective girth $\psi\ge C\m^2$, for some constant $C > 0$ to be chosen. Recall that we define
$$
\beta_k(G) = \max_{A\cap B = \emptyset}e(A, B) + k(|A| - |B|) - d(A).
$$
Suppose $2\lam_k(G) > n / \psi^{1/2}$. We prove that
$$
2\lam_k(G) \le \frac{n}{\psi^{1/2}} + \beta_k(G).
$$
If $\beta_k(G) < 0$ this is a contradiction, otherwise this provides an upper bound for $\lam_k(G)$. In either case we may conclude that
$$
2\lam_k(G) \le \frac{n}{\psi^{1/2}} + \max\{0, \beta_k(G)\}.
$$
Let $F$ be a $k$-matching with $n - |F| = \lam_k(G)$, and let
$$
X_0 = X_F = \{v : d_F(v) < k\}.
$$
Recursively define, for $j\ge 1$,
\begin{itemize}
\item $X_{2j-1}$ as the set of $v\notin X_0\cup X_1\cup \dots \cup X_{2j-2}$ such that there exists an edge $uv\notin F$ with $u\in X_{2j-2}$,
\item $X_{2j}$ as the set of $v\notin X_0\cup X_1\cup \dots \cup X_{2j-1}$ such that there exists an edge $uv\in F$ with $u\in X_{2j-1}$.
\end{itemize}
For $r \ge 1$ define a vertex partition by
\al{
  A_r & = X_0 \cup X_2 \cup \dots \cup X_{2r}, \\
  B_r & = X_1 \cup X_3 \cup \dots \cup X_{2r-1}, \\
  C_r & = V \setminus (A_r\cup B_r).
}
We make a sequence of claims about this partition.
\begin{claim}\label{cl:a} Any edge contained in $A_r$ is either in $F$ or on a cycle of length at most $4r + 1$.
\end{claim}
\begin{claim}\label{cl:b} Every $v\in A_{r}$ has at least $d(v)-k +1$ neighbours in $A_r\cup B_{r+1}$, and any edge incident to $v$ which is not contained in $A_r\cup B_{r+1}$ is in $F$.
\end{claim}
\begin{claim}\label{cl:c} Suppose $v\in B_r$. If $vw\in F$ then either $w\in A_{r}$ or $vw$ is on a cycle of length at most $4r-1$.
\end{claim}

\begin{proof}[Proof of Claim~\ref{cl:a}]
Suppose $uv\notin F$ is contained in $A_r$. Let $i, j\ge 0$ be such that $u\in X_{2i}, v\in X_{2j}$. By construction of the $X_\ell$, there exists a sequence of vertices $u_\ell\in X_\ell$, $\ell=0,1,\dots,2i$ with $u_{2i} = u$ such that $\{u_\ell, u_{\ell+1}\}\in G\setminus F$ for even $\ell$ and $\{u_\ell,u_{\ell+1}\}\in F$ for odd $\ell$. This is an $F$-alternating path. Likewise there exists an $F$-alternating path $(v_0,v_1,\dots,v_{2j}=v)$ from $X_0$ to $v$. If the paths do not intersect, then $P = (u_0,u_1,\dots,u_{2i}, v_{2j},v_{2j-1},\dots,v_1,v_0)$ is an $F$-alternating path with $u_0u_1, v_0v_1\notin F$, and $F\triangle P$ is a $k$-matching larger than $F$, contradicting the maximality of $F$. Hence the paths $(u_0,\dots,u)$ and $(v_0,\dots,v)$ must intersect, which implies the existence of a cycle of length at most $2i+2j+1\le 4r + 1$ containing $uv$.
\end{proof}
\begin{proof}[Proof of Claim~\ref{cl:b}]
Suppose $u\in A_r$. If $u\in X_0$ then $d_F(u) < k$, and any edge $uv\notin F$ has $v\in X_1\subseteq B_r$. If $u\in X_{2i}$ for some $i > 0$ then there is at least one $uv\in F$ with $v\in X_{2i-1}$, and any edge $uw\notin F$ either has $w\in A_{i}\cup B_{i}$ or $w\in X_{2i+1}$.
\end{proof}
\begin{proof}[Proof of Claim~\ref{cl:c}]
Let $j\ge 1$ be such that $v\in X_{2j-1}$. We have $d_F(v) = k$ by definition, since otherwise $v\in X_0$.

Suppose $vw\in F$. If $w\notin X_0\cup X_1\cup \dots X_{2j-1}$ then $w\in X_{2j}\subseteq A_r$. If $w\in X_{2i-1}$ for some $i < j$ then $v$ would have been placed in $X_{2i}$, contradicting $v\in X_{2j-1}$. We conclude that either $w\in X_{2j-1}$ or $w\in A_r$. Finally, if $w\in X_{2j-1}$, suppose $(v_0,\dots,v_{2j-1} = v)$ and $(w_0, \dots,w_{2j-1} = w)$ are $F$-alternating paths. If they do not intersect then $P = (v_0,\dots,v,w,\dots,w_0)$  has $|F\triangle P| > |F|$, so the paths must intersect. We conclude that $vw$ is on a cycle of length at most $2(2j-1)+1 \le 4r - 1$.
\end{proof}

Now fix some $r\ge 1$. Suppose $e\in G\setminus F$ is an edge incident to $A_r$. If it is contained in $A_r$ then it is on a cycle of length at most $4r + 1$ by Claim~\ref{cl:a}. If it has one endpoint in $C_r$ then by Claim~\ref{cl:b}, $e\in X_{2r}\times X_{2r+1}$. We have
\al{
  d(A_r) - d_F(A_r) & \le 2c_{4r+1} + e(X_{2r}, X_{2r+1}) + e(A_r, B_r) - e_F(A_r, B_r),
}
and reordering terms,
\al{
  e(A_r, B_r) - d(A_r) \ge e_F(A_r, B_r) - d_F(A_r) - 2c_{4r+1} - e(X_{2r}, X_{2r+1}).
}
Since $X_F\subseteq A_r$ we have
$$
d_F(A_r) = k|A_r| - 2\lam_k(G).
$$
Since every $v\in B_r$ has $d_F(v) = k$, Claim~\ref{cl:c} implies that
$$
e_F(A_r, B_r) \ge k|B_r| - c_{4r+1}.
$$
We then have
\al{
  \beta_k(A_r, B_r) =\ & k(|A_r| - |B_r|) + e(A_r, B_r)  - d(A_r) \\
  \ge\ & k(|A_r| - |B_r|) + e_F(A_r, B_r) - d_F(A_r) \\
  & - 2c_{4r + 1}-e(X_{2r}, X_{2r+1}) \\
  \ge\ & 2\lam_k(G) - 3c_{4r+1} - e(X_{2r}, X_{2r+1}).
}
As $\beta_k(A_r, B_r) \le \beta_k(G)$, we conclude that
\al{
  2\lam_k(G) \le\ & 3c_{4r+1} + e(X_{2r}, X_{2r+1}) + \beta_k(G).
}
It remains to choose $r$ so that $3c_{4r+1} + e(X_{2r}, X_{2r + 1}) \le n/\psi^{1/2}$. First note that
$$
\frac{n}{\psi^{1/2}} < 2\lam_k(G) = \sum_{v\in X_0} (k - d_F(X_0)) \le |X_0|.
$$
Let
$$
\e = \frac{1}{4\m \psi^{1/2}}.
$$
Let $r$ be the smallest integer for which $d(A_r) \le (1 + \e)d(A_{r-1})$. Then
$$
2m \ge d(A_{r-1}) \ge (1+\e)^{r-1}d(A_0) \ge (1+\e)^{r-1} \d|X_0| \ge (1+\e)^{r-1} \frac{2n}{\psi^{1/2}},
$$
so $(1+\e)^{r-1} \le \m \psi^{1/2}$. Since $\e = (4\m\psi^{1/2})^{-1} \le 1$ we have
$$
r \le 1 + \frac{\ln (\m\psi^{1/2})}{\ln (1 + \e)} \le 2 \e^{-1}\ln(\m\psi^{1/2}) \le \frac{\psi}{5},
$$
where we use $\psi \ge (C\m)^2$ for some large enough constant $C$ in the latter bound. Then $4r + 1 \le \psi$ and
$$
c_{4r + 1} \le c_\psi \le \frac{n}{\psi} \le \frac{n}{6\psi^{1/2}}.
$$
Note that $d(A_\ell) \le (1+\e)d(A_{\ell-1})$ implies
$$
e(X_{2\ell}, X_{2\ell+1}) \le d(X_{2\ell}) \le \e d(A_{\ell-1})\le  2m\e = \frac{n}{2\psi^{1/2}}.
$$
We conclude that
$$
3c_{4r + 1} + e(X_{2r}, X_{2r + 1}) \le \frac{n}{\psi^{1/2}},
$$
finishing the proof.

\section{Proof of Lemma~\ref{lem:posaish}}\label{sec:posaish}

Lemma~\ref{lem:posaish} has two parts, stated individually in the subsections below.

Suppose $F$ is a $k$-matching in $G = (V, E)$. We define a {\em walk (from $u$ to $v$)} as a sequence $P = (u = v_0,v_1,\dots,v_\ell = v)$ where $v_iv_{i+1}\in E$ for $0\le i < \ell$, allowing for edges to be repeated. Say that $P$ is {\em $F$-alternating} if the following holds for all $1\le i < \ell$. Let $P_i = (v_0,v_1,\dots,v_i)$ and $F_i = F\triangle P_i$. This is the set of edges that appear an odd number of times in total in $F$ and the multiset $P_i$. If $v_{i-1}v_i\in F_{i-1}$ we require that $v_iv_{i+1}\notin F_i$, and if $v_{i-1}v_i\notin F_{i-1}$ we require that $v_iv_{i+1}\in F_i$.

An $F$-alternating walk $(v_0,\dots,v_\ell)$ with $v_0v_1\in F$ and $v_{\ell-1}v_\ell\notin F$ is said to be of type $\downarrow\uparrow$, where one may think of walking along an edge of $F$ as going down, and an edge not in $F$ as going up. We similarly define $F$-alternating walks of type $\downarrow\downarrow, \uparrow\downarrow, \uparrow\uparrow$, where the membership of $v_0v_1$ in $F$ and $v_{\ell-1}v_\ell$ in $F_{\ell-1}$ determine the directions of the arrows.

For $\lor, \land\in \{\downarrow, \uparrow\}$ define
$$
A_F^{\lor\land}(u) = \{v : \text{ there is an $F$-alternating walk of type $\lor\land$ from $u$ to $v$}\}.
$$
The empty walk is defined to be of type $\uparrow\downarrow$. Note that if $x\in X$ and $y\in Y_x$ then $A_F^{\uparrow\downarrow}(y)\setminus\{x\}\subseteq Y_x$. Indeed, if $P$ is $F$-alternating of type $\uparrow\downarrow$ from $y$ to $z\ne x$ then $P\triangle F$ is a $k$-matching with $d_{F\triangle P}(z) < k$ and $d_{F\triangle P}(x) < k$.

Now suppose $F$ is a 2-factor. Our main interest will be in finding $F$-alternating walks $P = (w,\dots,z)$ of type $\downarrow\downarrow$ such that $\k(F\triangle P) < \k_2(G)$. Given such a walk, we have $\k_2(G+wz) < \k_2(G)$, so $w\in W$ and $z\in Z_w$. Let $B_F^{\lor\land}(u)$ be the set of vertices $v$ such that there exists an $F$-alternating walk $P$ of type $\lor\land$ from $y$ to $v$ such that $\k(F\triangle P) < \k_2(G)$. 

\subsection{Proof of Lemma \ref{lem:posaish} (i)}

Let $k\ge 2$. Suppose $G$ is a graph with $\d(G) \ge k + 1$ and $\lam_k(G) > 0$. Let $x\in X$. Lemma~\ref{lem:posaish}~(i) claims that there exists a nonempty vertex set $S_x$ such that
$$
e(S_x) \ge \frac{13}{12}|S_x|
$$
and $|S_x\cap Y_x| \ge \frac13|S_x| - 1$. We now prove this claim.

Let $F$ be a $k$-matching with $\lam_k(F) = \lam_k(G) > 0$ and let $x\in X_F$ so that $d_F(x) < k$. If $x$ is the only vertex with $d_F(x) < k$ then we must have $d_F(x) \le k - 2$, since $\lam_k(F) > 0$. Pick some walk $P = (x, x', x'')$ with $xx'\notin F$ and $x'x''\in F$, and replace $F$ by $F\triangle (x, x', x'')$. We may then assume that $|X_F| \ge 2$.

Let $y\in X_F\setminus \{x\}$. We prove that the set
$$
S = A_F^{\uparrow\uparrow}(y) \cup A_F^{\uparrow\downarrow}(y)
$$
is as desired. We first show that $e(S)$ is large by bounding $d_S(v)$ from below for all $v\in S$.

Suppose $v\in A_F^{\uparrow\downarrow}(y)$ and let $P = (y=v_0,\dots,v_\ell=v)$ be an $F$-alternating walk of type $\uparrow\downarrow$ from $y$ to $v$. Let $H = F\triangle P$, and note that $d_H(v) \le k-1$. For any edge $vw\notin H$, the walk $(v_0,\dots,v_\ell,w)$ shows that $w\in A_F^{\uparrow\uparrow}(y)\subseteq S$. This shows that $d_S(v) \ge d(v) - (k-1) \ge  2$.

Now suppose $v\in A_F^{\uparrow\uparrow}(y)$, letting $P = (y = v_0,\dots,v_\ell = v)$ be $F$-alternating of type $\uparrow\uparrow$. Let $H = F\triangle P$, and note that $|H| = |F| + 1$ with $d_H(y) = d_F(y) + 1 \le k$, $d_H(v) = d_F(v) + 1$ while $d_H(u) = d_F(u)$ for all $u\notin \{y, v\}$. Since $|F| = \lam_k(G)$ and $|H| > |F|$, it must be that $H$ is not a $k$-matching, so we have $d_H(v) = k + 1$. All $k+1$ edges of $H$ incident to $v$ may be used to extend the walk $P$, which shows that $d_S(v) \ge d_H(v) \ge k+1$.

We conclude that
\al{
  e(S) & \ge \frac12 \sum_{v\in A_F^{\uparrow\downarrow}(y)\setminus A_F^{\uparrow\uparrow}(y)} d_S(v) + \frac12 \sum_{v\in A_F^{\uparrow\uparrow}(y)} d_S(v) \\
  & \ge |A_F^{\uparrow\downarrow}(y)\setminus A_F^{\uparrow\uparrow}(y)| + \frac{k+1}2|A_F^{\uparrow\uparrow}(y)|. \label{eq:degbd}
}
We now show that $e(S) \ge \frac{13}{12}|S|$. Suppose $v\in A_F^{\uparrow\downarrow}(y)\setminus (\{y\}\cup A_F^{\uparrow\uparrow}(y))$, and let $P = (y = v_0,\dots,v_\ell = v)$ be $F$-alternating of type $\uparrow\downarrow$. Suppose $P$ is minimal in the sense that $v\notin \{v_0,\dots,v_{\ell-1}\}$. We have $v_{\ell-1}v\in F_{\ell-1}$, and since $P$ is minimal we must have $v_{\ell-1}v\in F$. We conclude that $A_F^{\uparrow\downarrow}(y)\setminus(\{y\}\cup A_F^{\uparrow\uparrow}(y))\subseteq N_F(A_F^{\uparrow\uparrow}(y))$, so $|A_F^{\uparrow\downarrow}(y)\setminus A_F^{\uparrow\uparrow}(y)| \le k|A_F^{\uparrow\uparrow}(y)| + 1$. So
\al{
  e(S_x) & \ge |A_F^{\uparrow\downarrow}(y)\setminus A_F^{\uparrow\uparrow}(y)| + \frac{k}{6}|A_F^{\uparrow\uparrow}(y)| + \left(\frac{k+1}{2} - \frac{k}{6}\right)|A_F^{\uparrow\uparrow}(y)| \\
  & \ge |A_F^{\uparrow\downarrow}(y)\setminus A_F^{\uparrow\uparrow}(y)| + \frac13\bfrac{|A_F^{\uparrow\downarrow}(y)\setminus A_F^{\uparrow\uparrow}(y)| - 1}{2} + \frac76|A_F^{\uparrow\uparrow}(y)| \\
  & = \frac76|S_x| - \frac16.
}
We have $|S_x| \ge 2$ since $y$ is incident to at least one edge not contained in $F$, so
$$
e(S_x) \ge \frac76|S_x| - \frac16 = \frac{13}{12}|S_x| + \frac1{12}|S_x| - \frac16 \ge \frac{13}{12}|S_x|.
$$

Finally, suppose $v\in A_F^{\uparrow\uparrow}(y)\setminus A_F^{\uparrow\downarrow}(y)$ and let $P$ be a minimal $F$-alternating walk of type $\uparrow\uparrow$ from $y$ to $v$. Let $vw\in F$. Then $w\in A_F^{\uparrow\downarrow}(y)$. We conclude that $A_F^{\uparrow\uparrow}(y)\setminus A_F^{\uparrow\downarrow}(y)\subseteq N_F(A_F^{\uparrow\downarrow}(y))$, so
$$
|A_F^{\uparrow\uparrow}(y)\setminus A_F^{\uparrow\downarrow}(y)| \le k|A_F^{\uparrow\downarrow}(y)|
$$
We then have, since $A_F^{\uparrow\downarrow}(y)\setminus \{x\}\subseteq Y_x$,
$$
|S_x| = |A_F^{\uparrow\uparrow}(y)\setminus A_F^{\uparrow\downarrow}(y)| + |A_F^{\uparrow\downarrow}(y)| \le (k+1)|A_F^{\uparrow\downarrow}(y)| \le (k+1)(|S_x\cap Y_x| + 1).
$$

\subsection{Proof of Lemma~\ref{lem:posaish}~(ii)}

Suppose $\d(G) \ge 4$, $\lam_2(G) = 0$, $\k_2(G) > 1$, and that $G$ is connected. Let $w\in W$ (if $W=\emptyset$ there is nothing to prove). Lemma~\ref{lem:posaish}~(ii) claims that there exists a nonempty vertex set $T_w$ such that
$$
e(T_w) \ge \frac76|T_w|
$$
and $|T_w\cap Z_w| \ge \frac13|T_w|$.

We proceed similarly to the proof of part (i). We first show that there exists a 2-factor $F\subseteq G$ such that $B_F^{\downarrow\downarrow}(w)\ne\emptyset$. To see this, as $w\in W$ there exists some $z\ne w$ such that $\lam_2(G + wz) = 0$ and $\k_2(G + wz) < \k_2(G)$, and we can let $F'$ be a 2-factor in $G+wz$ with $\k(F') = \k_2(G + wz)$. Then $wz\in F'$, as otherwise $F'\subseteq G$, contradicting $\k_2(G+wz) < \k_2(G)$. Let $F''\subseteq G$ be some 2-factor. The symmetric difference $F''\triangle F'$ consists of disjoint cycles $C_1\cup\dots\cup C_\ell$, where $wz\in C_1$ without loss of generality. Let $F = F''\triangle (C_2\cup\dots \cup C_\ell)$. This is a 2-factor, and $P = C_1\setminus \{wz\}$ is $F$-alternating of type $\downarrow\downarrow$, showing that $z\in B_F^{\downarrow\downarrow}(w)$.

Let $w\in W$, fix some 2-factor $F$ with $B_F^{\downarrow\downarrow}(w)\ne \emptyset$ and define
$$
T_w = B_F^{\downarrow\downarrow}(w)\cup B_F^{\downarrow\uparrow}(w).
$$
We show that $T_w$ is as claimed. First suppose $v\in B_F^{\downarrow\downarrow}(w)$, and let $P = (w = v_0,\dots,v_\ell = v)$ be an $F$-alternating walk of type $\downarrow\downarrow$ with $\k(F\triangle P) < \k_2(G)$. Let $H = F\triangle P$, noting that $v$ has $d_H(v) = 1$. For any edge $uv\in G\setminus H$, the walk $P = (v_0,\dots,v_\ell,u)$ shows that $u\in B_F^{\downarrow\uparrow}(w)\subseteq T_w$. As $v_{\ell-1}\in T_w$, we have $d_T(v)\ge d(v) - 1 \ge 3$.

Now suppose $v\in B_F^{\downarrow\uparrow}(w)$, and suppose $P = (w = v_0,\dots,v_\ell = v)$ is of type $\downarrow\uparrow$ with $\k(F\triangle P) < \k_2(G)$. Note that $H = F\triangle P$ has $d_H(w) = 1, d_H(v) = 3$, while all other vertices have degree 2. The component of $H$ containing $v$ is then a union of a cycle containing $v$ and a path from $w$ to $v$. This means that two of the edges incident to $v$ can be appended to $P$ to form an $F$-alternating walk $P'$ of type $\downarrow\downarrow$ such that $\k(F\triangle P') \le \k(F\triangle P) < \k_2(G)$. We conclude that $d_T(v) \ge 2$. As in \eqref{eq:degbd},
\al{
  e(T_w) \ge |B_F^{\downarrow\uparrow}(w)\setminus B_F^{\downarrow\downarrow}(w)| + \frac32 |B_F^{\downarrow\downarrow}(w)|.
}
Suppose $v\in B_F^{\downarrow\uparrow}(w)\setminus B_F^{\downarrow\downarrow}(w)$ and let $P$ be a minimal walk of type $\downarrow\uparrow$ from $w$ to $v$. By the previous argument, at least one of the edges of $F$ incident to $v$ can be used to extend $P$ without increasing the component number, and we conclude that $v\in N_F(B_F^{\downarrow\downarrow}(w))$. We then have
\begin{equation}\label{eq:Brel}
|B_F^{\downarrow\uparrow}(w)\setminus B_F^{\downarrow\downarrow}(w)| \le 2|B_F^{\downarrow\downarrow}(w)|,
\end{equation}
and
\al{
  e(T_w) & \ge |B_F^{\downarrow\uparrow}(w)\setminus B_F^{\downarrow\downarrow}(w)| + \frac13|B_F^{\downarrow\downarrow}(w)| + \frac76|B_F^{\downarrow\downarrow}(w)|  \ge \frac76|T_w|.
}
We have $B_F^{\downarrow\downarrow}(w) \subseteq Z_w$, so \eqref{eq:Brel} implies
\al{
  |T_w| = |B_F^{\downarrow\uparrow}(w)\setminus B_F^{\downarrow\downarrow}(w)| + |B_F^{\downarrow\downarrow}(w)| \le 3|B_F^{\downarrow\downarrow}(w)| \le 3|T_w\cap Z_w|.
}

\section{Proof of Lemma~\ref{lem:Gnd}}\label{sec:Gndproof}

Lemma~\ref{lem:Gnd} is stated for $G_{n, \bp}$, but we choose instead to show the result for $G_{n, \bd}$, i.e. the graph on $V = \{v_1,\dots,v_n\}$ where $v_i = d_i^{(n)}$ is fixed for all $i$. The properties involved have the same probability of holding in the two models, as they are independent of vertex labels.

\subsection{The configuration model}

The configuration model was introduced by Bollob\'as \cite{Bollobas80} as a method for sampling $G_{n, \bd}$. Fix $\bd = (d(v_1),\dots,d(v_n))$ with $\sum d(v) = 2m$ for some integer $m$, and for each $v_i\in V_n$ let $\PP(v_i) = \{x_{i,1},\dots,x_{i,d(v_i)}\}$ be a set of $d(v_i)$ {\em configuration points}, and let $\PP = \bigcup_i \PP(v_i)$. Let $\s$ be a perfect matching of $\PP$, i.e. a partition $\PP = e_1\cup\dots \cup e_m$ of $\PP$ into $m$ parts of size $2$. Define $C_{n, \bd}$ as the (multi)graph on $V_n$ given by including an edge between $v_i$ and $v_j$ for every pair $x\in \PP(v_i), y\in \PP(v_j)$ with $xy\in \s$.

Suppose $\D\le n^{1/6}$ and let
$$
\lam = \frac{\sum_v d(v)(d(v)-1)}{2\sum_v d(v)}.
$$
Then for any multigraph property $\mathbb{P}$,
$$
\Prob{G_{n, \bd} \in \mathbb{P}} \le (1+o(1))e^{\lam(\lam+1)}\Prob{C_{n, \bd}\in \mathbb{P}},
$$
see \cite[Theorem 10.3]{FriezeKaronski}. For any light-tailed degree sequence we have $\lam < \infty$ (see \eqref{eq:power} below), so any property that holds with high probability in $C_{n, \bd}$ also holds with high probability in $G_{n, \bd}$.

The following bound will be used in the upcoming sections.
\begin{lemma}\label{lem:randomset}
Let $\bd$ be a light-tailed degree sequence, and suppose $S$ is a vertex set chosen uniformly at random from subsets of $V$ of size $s$. Then there exists a constant $\k > 0$ such that
$$
\E{\prod_{v\in S} \binom{d(v)}{2}} \le \k^s.
$$
\end{lemma}
\begin{proof}
Let $V = \{u_1,\dots,u_n\}$ where $d(u_1) \ge d(u_2) \ge \dots \ge d(u_n)$. Then for any $s \in \{1,\dots,n\}$,
$$
sd(u_s) \le d(u_1) + \dots + d(u_s) = D(s) \le Cn\bfrac{s}{n}^\a,
$$
so
$$
d(u_s) \le C\bfrac{n}{s}^{1-\a}.
$$
Let $S = \{v_1,\dots,v_s\}$ where $v_1$ is chosen uniformly at random from $V$, and in general $v_i$ is chosen uniformly at random from $V\setminus \{v_1,\dots,v_{i-1}\}$. Write $d_2(v) = d(v)(d(v)-1)$. Then, as $\a > 1/2$,
\begin{equation} \label{eq:power}
  \E{d_2(v_1)} = \frac1n \sum_{u\in V} d_2(u) 
   = \frac1n \sum_{s=1}^n d(u_s)^2  
   \le \frac1n \sum_{s=1}^nC^2 \bfrac{n}{s}^{2 - 2\a} 
   \le \k_0, 
\end{equation}
for some constant $\k_0$. Conditional on $\{v_1,\dots,v_{i-1}\} = V_{i-1}$, we have
\begin{multline}
  \E{d_2(v_i) \mid V_{i-1}} = \frac{1}{n-i+1} \sum_{u\in V\setminus V_{i-1}} d_2(u) \\
  = \frac{n}{n-i+1} \sum_{u\in V} \frac{d_2(u)}{n} - \sum_{u\in V_{i-1}} \frac{d_2(u)}{n-i+1} 
  \le \frac{\k_0n}{n-i+1}.
\end{multline}
We then have
\al{
  \E{\prod_{i=1}^s \binom{d(v_i)}{2}} & = 2^{-s}\E{\prod_{i = 1}^s \E{d_2(v_i) \mid d(v_1),\dots,d(v_{i-1})}} \\
  & \le \bfrac{\k_0}{2}^s \prod_{i = 1}^s \frac{n}{n-i+1} = \bfrac{\k_0}{2}^s e^{O(s^2/n)}.
}
We can choose $\k$ so that this is at most $\k^s$.
\end{proof}

\subsection{Sparsity of the fixed-sequence random graph}

Fix some $\th > 1$ throughout this section, and fix a degree sequence $\bd$. Recall that we define
$$
D(s) = \max_{|S| = s} \sum_{v\in S} d(v).
$$
Assume there exist constants $C > 0, 1/2 < \a < 1$ such that for all $s$,
\begin{equation}\label{eq:Dsbound}
D(s) \le Cn\bfrac{s}{n}^\a.
\end{equation}
We show that there exists a $\g = \g(\th)$ such that in the configuration multigraph $C_{n, \bd}$,
\al{
  \Prob{\exists S : |S| \le \g n, e(S) \ge \th|S|} = o(1).
}
We may assume that every $v\in S$ has $d_S(v)\ge 2$. Indeed, suppose some $v\in S$ has $d_S(v) = 1$. Then
$$
e(S\setminus \{v\}) - \th|S\setminus \{v\}| = e(S) - \th|S| - 1 + \th > 0.
$$
So we bound the probability that there exists a set $|S| < \g n$ with $e(S) \ge \th|S|$ and $d_S(v) \ge 2$ for all $v\in S$. Let $\SS$ be the family of such sets.

Let $t = \cl{\th s}$. If $S\in \SS$ there is a set $\QQ\subseteq \PP(S) = \cup_{v\in S} \PP(v)$ of $2t$ configuration points, such that $|\QQ\cap \PP(v)| \ge 2$ for all $v\in S$. The number of ways to choose $\QQ$ is at most
$$
\left(\prod_{v\in S} \binom{d(v)}{2} \right) \times \binom{d(S) - 2s}{2t - 2s}.
$$
We pick a matching of the points in $\QQ$ in one of $(2t)!! = (2t-1)(2t-3)\times\dots\times 3 \times 1$ ways, and a matching of $\PP\setminus \QQ$ in one of $(2m-2t)!!$ ways. For a fixed set $S$ of size $s$ we then have
\al{
  \Prob{S\in \SS} \le \left(\prod_{v\in S}\binom{d(v)}{2}\right)\binom{d(S)-2s}{2t-2s} \frac{(2t)!!(2m-2t)!!}{(2m)!!}.
}
We delay bounding the first factor, and consider the remaining factors first. Stirling's formula gives, for large $N$,
\al{
  (2N)!! = \frac{(2N)!}{2^NN!} & = \bfrac{2N}{e}^N \left(\sqrt{2} + o(1)\right) \le 2\bfrac{2N}{e}^N, \label{eq:doubfacas}
}
so
\al{
  \frac{(2t)!!(2m-2t)!!}{(2m)!!} & \le 4\bfrac{2t}{e}^t \bfrac{2m-2t}{e}^{m-t}\bfrac{e}{2m}^{m}(1+o(1)) \\
  & = 4\bfrac{t}{m}^{t} \exp\left\{(m-t) \ln\left(1 - \frac{t}{m}\right)\right\} \\
  & \le \bfrac{t}{m}^{t}e^{t^2 / m - t} = \bfrac{s}{n}^t e^{O(s)}. \label{eq:doubfacfrac}
}
Here we used $t^t = s^t e^{O(s)}$ and $m = \Theta(n)$. Now, \eqref{eq:Dsbound} and the well-known bound $\binom{n}{s} \le (ne/s)^s$ give 
\al{
  \binom{d(S)-2s}{2t-2s} & \le \bfrac{eD(s)}{2t-2s}^{2t -2s}  \\
  & \le \left(\frac{n}{s}\bfrac{s}{n}^\a\right)^{2t-2s}e^{O(s)}  \le \bfrac{s}{n}^{(2\a - 2)(t-s)} e^{O(s)}.
}
Using $m = \Theta(n)$ we conclude that for a fixed set $S$,
$$
\Prob{S\in \SS} \le \left(\prod_{v\in S}\binom{d(v)}{2}\right)\bfrac{s}{n}^{2s - t + 2\a(t-s)} e^{O(s)}.
$$
By Lemma~\ref{lem:randomset}, letting $\mathbb{S}$ denote a set chosen uniformly at random from sets of size $s$, we then have
\al{
  \Prob{\exists |S| = s : S\in \SS} & \le \binom{n}{s}\E{\prod_{v\in \mathbb{S}} \binom{d(v)}{2}} \bfrac{s}{n}^{2s - t + 2\a(t-s)} e^{O(s)} \\
  & \le \bfrac{s}{n}^{(2\a - 1)(t - s)} e^{O(s)}.
}
With $\vartheta = (2\a-1)(\th - 1) > 0$, there then exists a $\g = \g(\th, \a)$ for which
\al{
  \Prob{\exists |S| \le \g n : S\in \SS} & \le \sum_{s = 1}^{\g n} \bfrac{s}{n}^{\vartheta s}e^{O(s)} = o(1).
}

\subsection{Effective girth the fixed-sequence random graph}

Let $S = \{v_1,\dots,v_s\}$ be a fixed vertex set of size $s$, and let $G\sim G_{n, \bd}$. Let $\PP_i = \PP(v_i), i =1,\dots,s$ be the sets of configuration points of the $v_i$. If $S$ contains a cycle then there exist sets $Q_i\subseteq \PP_i$ of size $2$ such that all configuration points in $Q = Q_1\cup\dots\cup Q_s$ are matched to other points in $Q$. Let $X(S)$ denote the indicator variable for the event of $S$ containing such a $Q$. For a fixed $S$ of size $s$, the number of ways to pick the $Q_i$ is
$
  \prod_{v\in S} \binom{d(v)}{2}.
$
Then for a fixed $S$, applying \eqref{eq:doubfacas} as in the previous section,
\al{
  \E{X(S)} & \le \left[\prod_{v\in S} \binom{d(v)}{2}\right] \frac{(2s)!!(2m-2s)!!}{(2m)!!} \le \left[\prod_{v\in S} \binom{d(v)}{2}\right] \bfrac{s}{m}^s e^{O(s)}.
}
Let $X_s = \sum_{|S| = s} X(S)$. Then
$$
\E{X_s} \le \binom{n}{s}\E{\prod_{v\in \mathbb{S}}\binom{d(v)}{2}} \bfrac{s}{m}^s e^{O(s^2/m)},
$$
where $\mathbb{S}$ is a set of size $s$ chosen uniformly at random. By Lemma~\ref{lem:randomset} there exists a constant $\k > 0$ such that the expected product is bounded by $\k^s$. We conclude that there exists a constant $C_1 > 0$ such that for $s = o(n)$,
$$
\E{X_s} \le e^{C_1s}.
$$
The total number of vertices on cycles of length at most $\om$ is then
$$
Y_\om = \sum_{s = 3}^\om sX_s,
$$
with $\E{Y_\om} \le e^{C_2\om}$ for some $C_2 > 0$. Markov's inequality gives
$$
\Prob{Y_\om \ge e^{2C_2\om}} \le e^{-C_2\om}.
$$
For any $\om = o\bfrac{\ln n}{\ln\ln n}$ we then have $Y_\om \le n/\om$ whp, and we may conclude that $\psi(G) \ge \ln^{1/2}n$ whp.

\section{Rerandomizability of the fixed-profile random graph}\label{sec:rerandom}

The aim of rerandomization is to find a random subgraph $\G_{n, \bp} \subseteq G_{n, \bp}$ which has $\d(\G_{n, \bp}) = \d(\bp)$, and such that $G_{n, \bp}$ is obtained from $\G_{n, \bp}$ by adding random edges, at least some of which are nearly uniformly distributed. The core idea behind is to remove edges from $G_{n, \bp}$, randomly chosen among edges with both endpoints having degree strictly greater than $\d$. This process can be reversed, so that $G_{n, \bp}$ is recovered from $\G_{n, \bp}$ by adding random edges, and we will study the distribution of those edges.

Note that we work in the fixed-profile graph and not fixed-sequence graph. If edges were to be removed from the fixed-sequence graph $G_{n, \bd}$, one can easily tell exactly which vertices in the resulting $\G_{n, \bd}$ have had edges removed, and edges may only be added to those vertices if we hope to recover a graph with distribution $G_{n, \bd}$.

The results in this section hold for all sparse degree profiles with minimum degree $\d$, but if the profile is not linearly unbounded then the applications are limited. Consider for example the profile $\bp$ with $p_4^{(n)}=n/2, p_5^{(n)} = n / 2$, i.e. half of the vertices in $G_{n, \bp}$ have degree 4 and the rest have degree 5. If $\G_{n, \bp}$ has minimum degree 4, then both endpoints of any edge added to recover $G_{n, \bp}$ must have degree 4 in $\G_{n, \bp}$. This means there is a set of up to $n/2$ vertices to which edges will never be added. Then the edges we add to $\G_{n, \bp}$ are not near-uniform, at least not in any sense that would be useful to our proofs.

For technical reasons we will put an upper bound on degrees in edges we remove, and before fully explaining the procedure we introduce the concept of {\em cutoffs} for degree profiles.

\subsection{Cutoffs}

Suppose $\bp$ is a degree profile with minimum degree at least $\d$. For a sequence of tuples $\CD = (C_n, D_n)$ of integers $\d < C_n \le D_n$ define
\al{
  \mathrm{head}(\bp, \CD) & = \liminf_{n\to\infty} \sum_{d = \d+1}^{C_n}\frac{dp_{d}^{(n)}}{n}, \\
  \mathrm{body}(\bp, \CD) & = \liminf_{n\to\infty} \sum_{d = C_n}^{D_n} \frac{dp_d^{(n)}}{n}, \\ 
  \mathrm{tail}(\bp, \CD) & = \liminf_{n\to\infty} \sum_{d\ge D_n} \frac{dp_d^{(n)}}{n}.
}
Say that $\CD$ is a {\em cutoff} for $\bp$ if all three quantities are positive, and $\sup D_n < \infty$. Heuristically, a useful cutoff for our proofs is one with small body and tail.
\begin{lemma}\label{lem:cutoff}
Let $\m, \e > 0$. Suppose $\bp$ is a degree profile with minimum degree $\d$ and average degree at least $\d + \e$ and at most $\m$.
\begin{enumerate}[(i)]
\item There exists a cutoff for $\bp$.
\item If $\bp$ is linearly unbounded then there exists a cutoff $\CD$ with
$$
\mathrm{body}(\bp, \CD) + \mathrm{tail}(\bp,\CD) < \e.
$$
\end{enumerate}
\end{lemma}
\noindent The proof is left for Appendix~\ref{app:cutoff}. We remark that in many natural situations, one can simply let $C_n = D_n = D$ for some large degree $D$.

\subsection{The rerandomization procedure}

Let $\bp_0$ be a degree profile with cutoff $\CD = (C_n, D_n)$. For any graph $G$ define $\F(G) = \F_\CD(G)$ as the set of edges $e = uv$ with $d(u), d(v) \in \{\d+1,\dots,D_n\}$, and let $\f(G) = |\F(G)|$. Note that for any edge $e\in \F(G)$,
\begin{equation}\label{eq:phichange}
\f(G-e) \le \f(G) \le \f(G-e) + 2\d+1.
\end{equation}
The upper bound holds because any edge $uv\in \F(G)\setminus \F(G - e)$ must have $d_G(u) = \d+1$ and $u\in e$, or $d_G(v) = \d+1$ and $v\in e$.  

Suppose $G^0 \sim G_{n, \bp_0}$ is chosen uniformly at random from $\GG_{\bp_0}$. For $i\ge 0$ we let $G^{-i-1}$ be the graph obtained by choosing an edge $e_i\in \F(G^{-i})$ uniformly at random, and setting $G^{-i-1} = G^{-i} - e_i$. If $\F(G^{-i}) = \emptyset$ we let $G^{-j} = \bot_j$ for all $j > i$, where ``$\bot$'' signifies a failed attempt at rerandomization. Let $\m_i(G) = \Prob{G^{-i} = G}$, and let $\GG_i$ denote the set of graphs $G$ with $\m_i(G) > 0$. Let $\GG_i^* = \GG_i\cup\{\bot_i\}$, $\GG^* = \cup_{i \ge 0} \GG_i^*$ and $\GG = \cup_{i\ge 0} \GG_i$.

For $i > 0$, $G\in \GG_i$ and $e\notin G$ we define
\al{
  \nu_G(e) = \Prob{G^{-i+1} = G + e \mid G^{-i} = G}
}
We can view $\nu_G(\cdot)$ as a probability measure on $\binom{V}{2} \setminus E(G)$, or alternatively as a measure on $\GG_{i-1}$, where we identify $\nu_G(G+e) = \nu_G(e)$. The latter interpretation can be extended to $\GG_i^*$, by defining a measure for $\bot = \bot_i$ by
$$
\nu_{\bot}(H) = \Prob{G^{-i + 1} = H \mid G^{-i} = \bot_i}, \quad H\in \GG_{i-1}^*.
$$
The family $\nu = \{\nu_G : G\in \GG^*\}$ defines an $s$-round sprinkling scheme (see Section~\ref{sec:sprinkling}) for any $G\in \GG$. If we also define $\nu^t[\bot_s]$ by
$$
\Prob{\nu^t[\bot_s] = H} = \Prob{G^{-s + t} = H \mid G^{-s} = \bot_s},
$$
then for $t\le s$,
$$
\nu^t[G^{-s}] \stackrel{d}{=} G^{-s+t},
$$
meaning the two random graphs are equal in distribution. In particular, $\nu^r[G^{-r}] \stackrel{d}{=} G^0$.

This section is dedicated to proving the following lemma.
\begin{lemma}\label{lem:rerandom}
Let $\e > 0$. Let $\bp_0$ be a degree profile with cutoff $\CD$ satisfying
\al{
  \mathrm{body}(\bp_0, \CD) + \mathrm{tail}(\bp_0, \CD) < \bfrac{\e}{2}^{1/2} .\label{eq:bodytail}
}
There exists a constant $\r_0 = \r_0(\bp, \CD) > 0$ such that for every $0 < \r \le \r_0$, there exist constants $\t, \s > 0$ such that if $r = \r n$, whp $G^{-r} \ne \bot$ and $\nu$ is an $(\t, \e, \s n)$-sprinkler for $G^{-r}$.
\end{lemma}
\noindent By Lemma~\ref{lem:cutoff}, if $\bp_0$ is linearly unbounded there exists a cutoff satisfying \eqref{eq:bodytail} for any $\e > 0$. Together with Lemma~\ref{lem:rerandom}, this implies Proposition~\ref{prop:rerandom}.

Note that if $\bp$ is the degree profile of a graph $G$, and $e = uv$ is an edge with $d(u) = k, d(v) = \ell$, then the degree profile of $G-e$ is $\bp - \bi_k - \bi_\ell$, where we define $\bi_k$ as the vector with a $1$ in position $k$ and a $-1$ in position $k-1$, otherwise consisting of zeros only. The (random) degree profile $\bp$ of any graph $G\in \GG$ is then of the form
$$
\bp = \bp_0 - \sum_{d = \d + 1}^D g(d; G, \bp_0) \bi_d,
$$
where $g(d; G, \bp_0)$ is a unique sequence of non-negative integers satisfying $\sum_d g(d; G, \bp_0) = 2(e(G_0) - e(G))$. Define for graphs $G = (V, E)$ and integers $k,\ell$,
\al{
  V_k(G) & = \{v\in V : d(v) = k\}, \\
  E_{k, \ell}(G) & = (V_k \times V_\ell) \cap E, \\
  F_{k, \ell}(G) & = (V_k\times V_\ell) \setminus E.
}
Let $R(G)$ be the set of edges $e\notin G$ with $e\in \F(G + e)$. Then $R(G)$ is the set of edges $e\notin G$ with $\nu_{G}(e) > 0$. We have
$$
R(G) = \bigcup_{(k, \ell)\in Z(G)}F_{k-1,\ell-1}(G),
$$
where $Z(G)$ is the set of pairs $(k,\ell)\in \{\d+1,\dots,D_n\}^2$ such that either $k\ne \ell$ and $g(k; G, \bp_0), g(\ell; G, \bp_0)$ are both positive, or $k = \ell$ and $g(k; G, \bp_0) \ge 2$. We define a subset $R_1(G)\subseteq R(G)$ by 
$$
R_1(G) = \bigcup_{\substack{(k, \ell)\in Z(G) \\ k,\ell\le C_n}} F_{k-1,\ell-1}(G).
$$
  
We break the proof of Lemma~\ref{lem:rerandom} down into claims, proved in the subsequent sections. The claims are the following.
\begin{claim}\label{cl:nuG}
Let $\vf > 0$. If $G\in \GG_s$ has $\f(G) \ge \vf n$, then for any $e\in R(G)$,
$$
\nu_G(e) = (1+o(1)) \frac{\m_{s-1}(G + e)}{\sum_{f\in R(G)}\m_{s-1}(G+f)}.
$$
\end{claim}
\begin{claim}\label{cl:Gef}
Let $\vf, \s > 0$ and $s \ge \s n$. Suppose $G\in \GG_s$ has $g(d; G, \bp_0) \ge \vf n$ for all $d\in \{\d+1,\dots,C_n\}$ and $\f(G) \ge \vf n$. There exists a constant $c = c(\vf, \s, \CD) > 0$ such that for any $e\in R(G)$ and $f\in R_1(G)$,
$$
\m_{s-1}(G + f) \le c\m_{s-1}(G + e).
$$
\end{claim}
\begin{claim}\label{cl:adG}
There exists some $\r_0 = \r_0(\bp, \CD) > 0$ such that for every $0 < \r \le \r_0$ the following holds. There is a $0 < \vf < \r/2$ such that if $r = \r n$ then whp, $\f(G_r) \ge \vf n$ and $g(d; G^{-r}, \bp_0) \ge \vf n$ for all $d\in \{\d+1,\dots,C_n\}$.
\end{claim}

\begin{proof}[Proof of Lemma~\ref{lem:rerandom}]
As $\CD$ is a cutoff we have $\mathrm{body}(\bp_0, \CD) > 0$, and Claim~\ref{cl:adG} can be applied. Let $\r_0 > 0$ be as provided by the claim, and let $r = \r n$ for some $\r \in (0, \r_0]$. Let $\HH_r$ be the set of $G\in \GG_r$ with $g(d; G, \bp_0) \ge \vf n$ for all $\d < d \le C_n$ and $\f(G) \ge \vf n$, where $\vf > 0$ is chosen according to Claim~\ref{cl:adG} so that $G_r\in \HH_r$ whp.

Let $s = \vf n / 3$. For $G\in \HH_r$ let $\SS_G$ be the $s$-round support of $\nu$, i.e. the set of graphs $H$ with $\Prob{\nu^t[G] = H} > 0$ for some $t\le s$. For any $G\in \HH_r$ and $H\in \SS_G$ we have
$$
g(d; H, \bp_0) \ge g(d; G, \bp_0) - 2s \ge \frac\vf3n,
$$
and $\f(H) \ge \f(G) \ge \vf n$. Let $t\le s$ and suppose $H\in \SS_G\cap \GG_{r-t}$, so $e(H) - e(G) = t$. Claims~\ref{cl:nuG} and \ref{cl:Gef} then show that there exists a $c > 0$ such that for any $e\in R_1(H)$,
\al{
  \nu_H(e) & = (1+o(1)) \frac{\m_{r-t-1}(H + e)}{\sum_{f\in R(H)} \m_{r-t-1}(H + f)} \\
  & \ge (1+o(1)) \frac{\m_{r-t-1}(H+e)}{\sum_{f\in R(H)} c^{-1}\m_{r-t-1}(H+e)} \ge \frac{c}{n^2}.
}
Since $g(d; H, \bp_0) \ge \vf n - 2t \ge 2$ for all $\d < d \le C_n$, we have $(k, \ell)\in Z(H)$ for all pairs $k,\ell\in \{\d,\dots,C_n-1\}$. The set $\binom{V}{2}\setminus R_1(H)$ consists of the $O(n)$ edges of $H$ along with any non-edge incident to a vertex of degree at least $C_n$. This set has size at most
\al{
  \left|\binom{V}{2}\setminus R_1(H)\right| & \le O(n) + (\mathrm{body}(\bp_0, \CD) + \mathrm{tail}(\bp_0, \CD))^2 n^2  < \frac{\e}{2} n^2.
}
For any set $S\subseteq \binom{V}{2}$ of size at least $\e n^2$ we then have
\al{
  \nu_H(S) \ge \frac{c\left|S\cap R_1(H)\right|}{n^2} \ge \frac{c\e}{2}.
}
This shows that $\nu_H$ is $(\t, \e)$-random with $\t = c\e/2$.
\end{proof}

\subsection{Claim~\ref{cl:nuG}. An expression for $\nu_G(e)$}

Let $G\in \GG_s$ and $e\in R(G)$. Then
\al{
  \nu_G(e) & = \Prob{G_{s-1} = G + e \mid G_s = G} \\
  & = \frac{\Prob{G_{s-1} = G + e}}{\Prob{G_s = G}} \Prob{G_s = G \mid G_{s-1} = G + e} \\
  & = \frac{\m_{s-1}(G+e)}{\m_s(G)} \frac{1}{\f(G+e)}.
}
We also have
\al{
  \m_s(G) & = \sum_{f\notin G} \Prob{G_s = G\mid G_{s-1} = G + f}\m_{s-1}(G+f) \\
  & = \sum_{f\in R(G)} \frac{\m_{s-1}(G + f)}{\f(G + f)}.
}
If $\f(G) \ge \vf n$ for some $\vf > 0$, we have $\f(G + e) \le \f(G) + 2\d = \f(G)(1+O(1/n))$ for any $e\in R(G)$, so
\al{
  \nu_G(e) & = \frac{\m_{s-1}(G + e)}{\f(G + e)} \left(\sum_{f\in R(G)} \frac{\m_{s-1}(G+f)}{\f(G + f)}\right)^{-1} \\
  & = (1+o(1))\frac{\m_{s-1}(G+e)}{\sum_{f\in R(G)} \m_{s-1}(G+f)}. \label{eq:pGe}
}

\subsection{Claim~\ref{cl:Gef}. Comparing $\m_{s-1}(G+e)$ and $\m_{s-1}(G + f)$}\label{sec:compare}

For a graph $G\in \GG_s$, let $A(G)$ be the set of edge sequences $(e_j, 1\le j\le s)$ such that $G + e_s + \dots + e_1$ has degree profile $\bp_0$. Then
\al{
  \m_s(G) & = \frac{1}{|\GG_{n,\bp_0}|}\sum_{(e_j)\in A(G)} \prod_{j = 1}^s \frac{1}{\f(G + e_s + e_{s-1} + \dots + e_j)}.
}
For edge sequences $(e_j)\in A(G)$ define
$$
\pi_G((e_j)) = \prod_{j=1}^s \frac{1}{\f(G + e_s + \dots + e_j)}.
$$
Suppose $(e_j)\in A(G)$, and write $e_j = \{x_{2j}, x_{2j-1}\}$ for $j=1,\dots,s$, arbitrarily picking which endpoint is $x_{2j}$ and which is $x_{2j-1}$. For $Z\subseteq \{x_1,\dots,x_{2s}\}$ and $(f_j)\in A(G)$, say $(e_j)$ and $(f_j)$ are equal modulo $Z$ if there exists a labelling $f_j = \{y_{2j}, y_{2j-1}\}$ such that $x_i = y_i$ whenever $x_i\notin Z$, and $y_i\in Z$ whenever $x_i\in Z$. For $(e_j), (f_j)\in A(G)$ we define a distance function by
\al{
  \dist((e_j), (f_j)) = \min\{|Z| : (e_j) = (f_j)\mod Z\}.
}

\begin{subclaim}\label{cl:piG}
Let $\vf > 0$. Suppose $s\le n$ and $\f(G) \ge \vf n$. There exists a constant $c_1 = c_1(\vf, D_n)$ such that if $(e_j), (f_j)\in A(G)$ have $\dist((e_j), (f_j)) \le 6$ then
\al{
  \pi_G((e_j)) & \le c_1 \pi_G((f_j)).
}
\end{subclaim}
\begin{proof}[Proof of Subclaim~\ref{cl:piG}]
As noted in \eqref{eq:phichange}, for any $e\in R(G)$,
\begin{equation}\label{eq:phigbd}
\f(G) \le \f(G + e) \le \f(G) + 2\d + 1.
\end{equation}
Next note that if $(e_j) = (f_j)$ modulo $Z$ where $|Z| \le 6$, then the graphs $G_j = G + e_s + \dots + e_j$ and $H_j = G + f_s + \dots + f_j$ can be written as $G_j = H \cup E$ and $H_j = H \cup F$ for some graph $H$ and edge sets $|E| = |F| \le 6(D_n-\d)$. Indeed, suppose $x_i, y_i$ is a labelling of the sequences with $x_i = y_i$ whenever $x_i$ or $y_i$ is not in $Z$. Then any $w\in Z$ appears at most $D_n - \d$ times in either sequence, so the sequences differ in at most $6(D_n-\d)$ edges. We conclude that
$$
\f(G_j) \le \f(H) + 6(D_n-\d)(2\d + 1) \le \f(H_j) + 12D_n^2.
$$
Define $G_j = G + e_s + \dots + e_j$ and $H_j = G + f_s + \dots + f_j$ for $j \in \{1,\dots,s\}$. Then $\min_j \f(H_j) \ge \f(G)$, and
\begin{multline}
  \pi_G((e_j)) = \prod_{j = 1}^s \frac{1}{\f(G_j)} \ge \prod_{j = 1}^s \frac{1}{\f(H_j) + 12D_n^2} \\
  \ge \prod_{j = 1}^s \frac{1}{\f(H_j)\left(1 + \frac{12D_n^2}{\min \f(H_j)}\right)} 
  \ge \left(1 + \frac{12D_n^2}{\f(G)}\right)^{-s} \pi_G((f_j)).
\end{multline}
By assumption we have $\f(G) \ge \vf n$ for some $\vf > 0$, so
\al{
  \left(1 + \frac{12D_n^2}{\f(G)}\right)^{-s} \ge \left(1 + \frac{12D_n^2}{\vf n}\right)^{-n} \ge \exp\left\{- \frac{12D_n^2}{\vf}\right\}.
}
\end{proof}

\begin{subclaim}\label{cl:atob}
Let $\vf > 0$ and $d\in \{\d+1,\dots,C_n\}$. Suppose $g(d; G, \bp_0) \ge \vf n$, and that $g(D; G, \bp_0) = 0$ for all $D > D_n$. Then for any $H$ with degree profile $\bp_0$ such that $G\subseteq H$, it holds that
\al{
  |\{v : d_G(v) < d \le d_H(v) \le D_n\}| \ge \vf n.
}
\end{subclaim}

\begin{proof}[Proof of Subclaim~\ref{cl:atob}]
The degree profile $\bp = (p_a : a\ge \d)$ of $G$ is
$$
\bp = \bp_0 - \sum_{a = \d+1}^{D_n} g(a; G, \bp_0)\bi_a.
$$
Write $\bp(a) = p_a$. For any $a\in \{\d, \dots, D-1\}$ this means that
$$
\bp(a) = \bp_0(a) + g(a + 1, G, \bp_0) - g(a, G, \bp_0),
$$
with $g(\d, G, \bp_0) = 0$. In particular,
$$
\sum_{a = \d}^{d-1} \bp(a) = g(d; G, \bp_0) + \sum_{a = \d}^{d-1} \bp_0(a).
$$
Suppose $H\supseteq G$ has degree profile $\bp_0$. Then there must be $g(d; G, \bp_0) \ge \vf  n$ vertices with degree at  most $d-1$ in $G$ and at least $d$ in $H$. Since $g(D; G, \bp_0) = 0$ for $D > D_n$, any vertex with $d_G(v) < d_H(v)$ must have $d_H(v) \le D_n$.
\end{proof}
\begin{subclaim}\label{cl:Nf}
Let $e\in R(G)$ and $f\in R_1(G)$. Let $\s > 0$ and suppose $s \ge \s n$. Suppose there exists a $\vf > 0$ such that $g(d; G, \bp_0) \ge \vf n$ for all $d\in \{\d + 1,\dots,C_n\}$. Then there exists a constant $c_2 = c_2(\vf, \s, \CD) > 0$ such that for any $(e_j)\in A_e(G)$ and $(f_j) \in A_f(G)$,
\al{
  |N_6^e((f_j))| & \le c_2n^2,  \label{eq:Nkf1}\\
  |N_6^f((e_j))| & \ge c_2^{-1}n^2, \quad \text{if $e_s\cap f = \emptyset$}, \label{eq:Nkf2} \\
  |N_3^e((f_j))| & \le c_2n, \label{eq:Nkf3}\\
  |N_3^f((e_j))| & \ge c_2^{-1}n, \quad \text{if $|e_s\cap f| = 1$}. \label{eq:Nkf4}
}
\end{subclaim}
\begin{proof}[Proof of Subclaim~\ref{cl:Nf}]
We first note that if $(e_j) \in A(G)$ and $Z$ is of size at most 6, then the number of $(f_j)\in A(G)$ with $(f_j) = (e_j)$ modulo $Z$ is bounded above by $6^{6D_n}$. Indeed, any such $(f_j)$ is obtained from $(e_j = \{x_{2j}, x_{2j-1}\})$ by locating the indices $i$ where $x_i\in Z$, and replacing $x_i$ by one of the six vertices of $Z$. There are at most $6D_n$ such indices.

{\em Proof of \eqref{eq:Nkf1}.} Suppose $(e_j), (f_j)\in A(G)$ where $f_s = f$, and $Z$ is such that $(e_j) = (f_j)$ modulo $Z$. Then $e_s\triangle f \subseteq Z$. There are at most $n^2$ choices for the remaining two vertices of $Z$, and for each $Z$ there are at most $6^{6D_n}$ choices of $(f_j)\in A(G)$. We conclude that
$$
|N_6^f((e_j))| \le 6^{6D_n}n^2.
$$

{\em Proof of \eqref{eq:Nkf3}.} This follows similarly. Suppose $|Z| = 3$ and $(e_j) = (f_j)$ modulo $Z$. We have $|e_s\triangle f| = 2$ and $e_s\triangle f\subseteq Z$. There are $O(n)$ ways of choosing an additional vertex, and $O(1)$ sequences equal to $(e_j)$ modulo $Z$. Then
$$
|N_3^f((e_j))| \le 3^{3D_n}n.
$$

{\em Proof of \eqref{eq:Nkf4}.} Suppose $e_s = \{u_1, v\}$ and $f = \{u_2, v\}$ where $u_1\ne u_2$. Arbitrarily label the endpoints of the $(e_j)$ edges as $e_j = \{x_{2j}, x_{2j-1}\}$ for $j=1,\dots,s$. We construct a sequence $(f_j = \{y_{2j}, y_{2j-1}\})$ which is equal to $(e_j)$ modulo some set $Z$ of size 4, such that $H' = G + f_s + \dots + f_1$ has the same degree profile as $H = G + e_s + \dots + e_1$.

Let $X = \{x_1,\dots,x_{2s}\}$, and for $x\in X$ let $I(x)$ be the set of $i$ for which $x_i = x$.

{\bf Case 1.} Suppose $u_2\in X$. Let $Z = \{u_1, u_2, u_3\}$, where $u_3$ is chosen arbitrarily from $X\setminus\{u_1,u_2, v\}$. We construct $(f_j = \{y_{2j}, y_{2j-1}\})$ as follows.
\begin{enumerate}
\item Let $y_{2s} = u_2$.
\item Pick some $i\in I(u_2)$ and let $y_i = u_3$.
\item Pick some $i\in I(u_3)$ and let $y_i = u_1$.
\item For any unassigned $i\in \{1,\dots,2s\}$ let $y_i = x_i$.
\end{enumerate}
The resulting graph $H'$ has $\bp(H') = \bp(H)$, and the number of choices involved is at least $|X| - 3$ (the number of choices for $u_3$). Any vertex appears at most $D_n$ times in the sequence $x_1,\dots,x_{2s}$, so $|X| \ge 2s/D_n = \OM(n)$.

{\bf Case 2.} Suppose $u_2\notin X$. Then, as $f\in R_1(G)$, we have $d_H(u_2) = d_G(u_2) < C_n$. Let $u_3$ be a vertex with $d_G(u_3) \le d_G(u_2)$ and $d_H(u_3) > d_H(u_2)$. This can be chosen in $\OM(n)$ ways by Subclaim~\ref{cl:atob}. Note that $|I(u_3)| = d_H(u_3) - d_G(u_3) \ge d_H(u_3) - d_H(u_2)$.

Construct $(f_j = \{y_{2j}, y_{2j-1}\})$ by the following procedure.
\begin{enumerate}
\item Let $y_{2s} = u_2$.
\item Pick some $i(u_3)\in I(u_3)$ and let $y_{i(u_3)} = u_1$.
\item Pick a set $J(u_3)\subseteq I(u_3) \setminus \{i(u_3)\}$ of size $d_H(u_3) - d_H(u_2) - 1$, and let $y_j = u_2$ for all $j\in J(u_3)$.
\item For any unassigned $i$ let $y_i = x_i$.
\end{enumerate}
Let $d, d'$ denote degrees in $H, H'$ respectively. Then this process gives
\al{
  d'(u_1) & = d(u_1) - 1 + 1 = d(u_1), \\
  d'(u_2) & = d(u_2) + 1 + (d(u_3) - d(u_2) - 1) = d(u_3), \\
  d'(u_3) & = d(u_3) - 1 - (d(u_3) - d(u_2) - 1) = d(u_2), \\
}
This shows that $\bp(H') = \bp(H)$.

{\em Proof of \eqref{eq:Nkf2}.} Suppose $e = \{u_1, v_1\}$ and $f = \{u_2,v_2\}$ with $e\cap f = \emptyset$. Let $(e_j) \in A_e(G)$. Pick $u_3$ as above, and pick $v_3$ similarly, for a total of $\OM(n^2)$ choices. We can then construct $(f_j)\in N_6^f((e_j))$ as above, doing steps 1--4 (in either case) separately for the $u_i$ and for the $v_i$.
\end{proof}

Note that $N_k$ membership is reflexive in the sense that $(f_j)\in N_k^f((e_j))$ if and only if $(e_j)\in N_k^e((f_j))$, where $e = e_s$ and $f = f_s$. We have
\al{
  \m_{s-1}(G + f) & = \frac{\f(G + f)}{|\GG_{n, \bp_0}|} \sum_{(f_j)\in A_f(G)} \pi_G((f_j)) \\
  & = \frac{\f(G + f)}{|\GG_{n, \bp_0}|} \sum_{(f_j)\in A_f(G)}  \sum_{(e_j)\in N_k^e((f_j))} \frac{\pi_G((f_j))}{|N_k^e((f_j))|} \\
  & = \frac{\f(G + f)}{|\GG_{n, \bp_0}|} \sum_{(e_j)\in A_e(G)} \sum_{(f_j)\in N_k^f((e_j))} \frac{\pi_G((f_j))}{|N_k^e((f_j))|}. \label{eq:justastep}
}
Subclaim~\ref{cl:piG} shows that $\pi_G((f_j)) \le c_1\pi_G((e_j))$ for any $(f_j)\in N_k^f((e_j))$. Thus \eqref{eq:justastep} is bounded above by
\al{
  c_1\frac{\f(G + f)}{|\GG_{n, \bp_0}|} \sum_{(e_j)\in A_e(G)} \pi_G((e_j)) \sum_{(f_j)\in N_k^f((e_j))} \frac{1}{|N_k^e((f_j))|)} \\
}
From Subclaim~\ref{cl:Nf} we have $|N_k^e((f_j))| \le c_2n^{k/2}$ for all $(f_j)\in A_f(G)$ and $|N_k^f((e_j))| \ge c_2^{-1}n^{k/2}$ for all $(e_j)\in A_e(G)$, so there exists a constant $c > 0$ such that for any $(e_j)\in A_e(G)$,
$$
\sum_{(f_j)\in N_k^f((e_j))} \frac{1}{|N_k^e((f_j))|} \le c_2^2.
$$
We conclude that
\al{
  \m_{s-1}(G + f) & \le c_1c_2^2  \frac{\f(G + f)}{|\GG_{n, \bp_0}|} \sum_{(e_j)\in A_e(G)} \pi_G((e_j)) \\
  & = c_1c_2^2 \frac{\f(G+f)}{\f(G+e)} \m_{s-1}(G + e).
}
We have $\f(G) = \OM(n)$ by assumption, so $\f(G + f) / \f(G+e) = 1 + o(1)$. This finishes the proof.

\subsection{Proof of Claim~\ref{cl:adG}.}\label{sec:Gs}

Claim~\ref{cl:adG} essentially states that if we remove $r = \OM(n)$ edges, then whp $g(d; G^{-r}, \bp_0) = \OM(n)$ for all $d\in \{\d+1,\dots,C_n\}$. We will show that there are $\OM(n)$ vertices that have degree at least $C_n$ in $G^0$, and degree $\d$ in $G^{-r}$. We then show how this implies the claim.

Suppose $\bp_0 = (p_d)$ is a profile with cutoff $\CD$. Let $r = \r n$ for some $\r > 0$ to be determined. Let $d_i(v)$ denote the degree of a vertex in $G^{-i}$, and define a vertex set
$$
U = \{u: C_n \le d_0(u) \le D_n \text{ and } d_s(u) = \d\}.
$$
We claim that with high probability, $|U| = \OM(n)$. Define $d_0^\F(u)$ as the number of edges of $\F(G^0)$ incident to $u$ in $G^0$. Let
$$
W = \{u : C_n \le d_0(u) \le D_n \text{ and } d_0^\F(u) \ge d_0(u)-\d\},
$$
and note that $U\subseteq W$. Indeed, $\F(G^{-i}) \subseteq \F(G^{0})$ for all $i$, so $d_s(u) \ge d_0(u) - d_0^\F(u)$.

We first show that $|W| \ge \xi n$ whp, for some constant $\xi = \xi(\bp_0, \CD)$. Note first of all that the number of vertices $u$ with $C_n \le d_0(u) \le D_n$ is at least $\mathrm{body}(\bp_0, \CD)n = \OM(n)$ as $\CD$ is a cutoff. Fix a degree sequence $\bd$ with profile $\bp_0$ and let $V_\CD$ be the set of vertices $u$ with $C_n \le d(u) \le D_n$ under $\bd$. Consider revealing $G_{n, \bd}$ using the configuration model as follows. Let $\PP(u)$ be the set of configuration points for any vertex, and let $\s$ denote the random matching of $\PP(V) = \cup_{u\in V} \PP(u)$. Let $\QQ$ be the set of configuration points belonging to vertices with degree in $\{\d+1,\dots,D_n\}$, and note that
$$
|\QQ| \ge \d(\mathrm{head}(\bp_0, \CD) + \mathrm{body}(\bp_0, \CD))n = \OM(n).
$$
Suppose $\s$ has been partially revealed, i.e. condition on $\s$ agreeing with a matching $\s_0$ of some $\PP_0 \subseteq \PP$. Denote this as $\s_0\subseteq \s$. Then for any $x\in \PP(V_\CD)\setminus \PP_0$,
$$
\Prob{\s(x) \in \QQ \mid \s_0 \subseteq \s} \ge \frac{|\QQ| - 2|\PP_0|}{2m - 2|\PP_0| + 1}.
$$
As long as $|\PP_0| \le \frac14|\QQ|$, this probability is bounded below by $p = |\QQ| / 8m  = \OM(1)$.

Let $u_1\in V_\CD$ and reveal $\s(x)$ for all $x\in \PP(u_1)$. For each $x$ we have $\s(x)\in \QQ$ with probability at least $p = \OM(1)$. We then have probability $q = \OM(1)$ of at least $d(u_1) - \d$ of the $d(u_1)\le D_n = O(1)$ points picking in $\QQ$, which means $u_1\in W$. (Note that $\PP(u_1)\subseteq \QQ$, and $\s(x)\in \PP(u_1)$ for some $x$ only makes this more likely).

Let $i > 1$ and pick some $u_i \notin\{u_1,\dots,u_{i-1}\}$. For $x\in \PP(u_i)$, if $\s(x)$ has already been revealed it means $\s(x) \in \PP(u_j) \subseteq \QQ$ for some $j < i$. This only increases the probability of $u_i\in W$. As long as $D_ni \le \frac14|\QQ|$ we have probability at least $q$ of $u_i\in W$, conditional on what has been revealed previously. Letting $\ell = \frac{|\QQ|}{4D_n} = \OM(n)$ we conclude that
$$
\Prob{|W| < \frac{q}{2}\ell} \le \Prob{\mathrm{Bin}(\ell, q) < \frac{q}{2}\ell} = o(1).
$$
We let $\xi = \xi(\bp, \CD) = q\ell / 2$ and $\r_0 = \xi / (8\d + 4)$.

Suppose $G^0$ is a graph with $|W| \ge \xi n$ for some $\xi > 0$. Note that
$$
\f(G^0) = |\F(G^0)| \ge \frac12\sum_{u\in W} d_0^\F(u) \ge \frac{\xi}{2}n.
$$
Let $r  = \r n$ for some $0 < \r \le \frac{\xi}{8\d+4}n$. Then, as $\f(G) \le \f(G-e) + 2\d+1$ for any $e\in \F(G)$ (see \eqref{eq:phigbd}),
$$
\f(G^{-r}) \ge \f(G^0) - r(2\d+1) \ge \frac{\xi}{4}n,
$$
proving part of the claim.

We now show that, conditional on $|W| \ge \xi n$ for some $\xi > 0$, we have $|U| = \OM(n)$ whp. Let
$$
\F(G^0) = (e_1,e_2,\dots,e_{\f(G^0)})
$$
be a permutation of the edges of $\F(G^0)$ chosen uniformly at random. We can then view the process $G^0, G^{-1},\dots,G^{-r}$ as starting with $G^0 = G$, and removing the edges $e_1,e_2,\dots$ as long as they remain in $\F$, stopping when $r$ edges have been removed.

For $u\in W$ let $F_u$ be a set of $d(u)-\d$ edges of $\F(G^0)$ incident to $u$. Let $F_u^2$ be the set of edges of $\F(G^0)$ which are not incident to $u$, but intersect some edge of $F_u$. If all edges of $F_u$ appear before all edges of $F_u^2$ in the sequence $e_1,\dots,e_\f$, and $F_u\subseteq \{e_1,\dots,e_r\}$, then we have $u\in U$. Since $|F_u^2| \le D_n^2 = O(1)$, we have
$$
\Prob{\text{$F_u$ appears before $F_u^2$}} = \binom{|F_u| + |F_u^2|}{|F_u|}^{-1} = \OM(1).
$$
We also have, since $r = \OM(n)$,
\al{
  \Prob{F_u\subseteq \{e_1,\dots,e_r\}} & \ge \Prob{F_u\cup  F_u^2 \subseteq \{e_1,\dots,e_r\}} \\
  & \ge \frac{\binom{r}{|F_u| + |F_u^2|}}{\binom{\f(G^0)}{|F_u| + |F_u^2|}} = \OM(1).
}
The events $\{F_u\cup  F_u^2\subseteq \{e_1,\dots,e_r\}\}$ and $\{\text{$F_u$ appears before $F_u^2$}\}$ are independent, and we conclude that
$$
\Prob{u\in U} = \OM(1), \quad u\in W.
$$
Note that we revealed the location in $(e_1,\dots,e_\f)$ of $|F_u| + |F_u|^2 \le D_n^2$ edges in this process. We can repeat the argument, picking some vertex $u'$ with $F_{u'}\cup F_{u'}^2$ not intersecting $F_u\cup F_u^2$, and concluding that $\Prob{u'\in U} = \OM(1)$. This can be repeated for $\OM(n)$ vertices $w\in W$, each time with probability $\OM(1)$ of $w\in U$, conditional on the previous samples. We conclude that $|U| = \OM(n)$ whp.

We now argue that $g(d; G^{-s}, \bp_0) \ge |U|$ for all $d\in \{\d + 1,\dots,C_n\}$. This is not hard to see: if $G^{-i}$ is obtained from $G^{-i + 1}$  by removing an edge $e_i$ with $t\in \{0, 1, 2\}$ endpoints of degree $d$ then
$$
g(d; G^{-i}, \bp_0) = g(d; G^{-i+1}, \bp_0) + t.
$$
For each vertex $u\in U$ there exists some $i \le s$ for which $u\in e_i$ and $d_{i-1}(u) = d$, so
$$
g(d; G^{-r}, \bp_0) \ge |U|.
$$

\section{Concluding remarks}\label{sec:concluding}

We have shown, under some conditions on $\bd$, that determining Hamiltonicity in $G_{n, \bd}$ comes down to determining the likely value of
$$
\beta_2(G_{n, \bd}) = \max_{A\cap B = \emptyset} e(A, B) + 2(|A| - |B|) - d(A).
$$
We conclude the paper by discussing the necessity of two of our conditions on $\bd$.

Firstly, we assumed that $\bd$ is linearly unbounded purely in order to allow for rerandomization. There is no reason to believe this to be necessary for the result, and we conjecture that Theorem~\ref{thm:main} still holds without this condition.

Secondly, we assume $\d\ge 4$, where a clear goal (motivated by the random regular graph) would be to prove the same result for $\d\ge 3$. The minimum degree was used at two points: in Lemma~\ref{lem:StoB} and in Lemma~\ref{lem:posaish}~(ii). For the question of finding a $k$-factor with $\d\ge k+1$, only Lemma~\ref{lem:StoB} of the two is relevant. We do not believe Lemma~\ref{lem:StoB} to be a large obstruction, but spent little effort on it given the seemingly larger challenge of improving Lemma~\ref{lem:posaish}~(ii). The problem here is that while we can show that $e(T_w) > |T_w|$ when $\d = 3$, our methods require $e(T_w) \ge (1 + \e)|T_w|$ for some constant $\e > 0$. Frieze and Pittel encountered the same problem in the context of $G_{n, cn}$ conditioned on having minimum degree 3 \cite[Remark 4.1]{FriezePittel}.

\bibliographystyle{plain}
\bibliography{config_arxiv_200115}

\begin{thebibliography}{10}

\bibitem{AnastosFrieze19}
M.~Anastos and A.~M. Frieze.
\newblock Hamilton cycles in random graphs with minimum degree at least 3: an
  improved analysis.
\newblock {\em ArXiv e-prints}, 1906.01433, May 2019.

\bibitem{BohmanFrieze09}
T.~Bohman and A.~M. Frieze.
\newblock {H}amilton cycles in 3-out.
\newblock {\em Random Structures \& Algorithms}, 35(4):393--417, 2009.

\bibitem{BollobasCooperFennerFrieze}
B.~Bollob{\'{a}}s, C.~Cooper, T.~Fenner, and A.~M. Frieze.
\newblock On {H}amilton cycles in sparse random graphs with minimum degree at
  least $k$.
\newblock {\em Journal of Graph Theory}, 34:42--59, 2000.

\bibitem{Bollobas80}
B{\'{e}}la Bollob{\'{a}}s.
\newblock A probabilistic proof of an asymptotic formula for the number of
  labelled regular graphs.
\newblock {\em Eur. J. Comb.}, 1(4):311--316, 1980.

\bibitem{CooperFriezeKrivelevich}
C.~Cooper, A.~M. Frieze, and M.~Krivelevich.
\newblock Hamilton cycles in random graphs with a fixed degree sequence.
\newblock {\em SIAM Journal on Discrete Mathematics}, 24:558 -- 569, 2010.

\bibitem{CooperFriezeReed}
C.~Cooper, A.~M. Frieze, and B.~Reed.
\newblock Random regular graphs of non-constant degree: Connectivity and
  {H}amiltonicity.
\newblock {\em Combinatorics, Probability and Computing}, 11(3):249–261,
  2002.

\bibitem{ErdosRenyi60}
P.~Erd{\H{o}}s and A~R{\'{e}}nyi.
\newblock On the evolution of random graphs.
\newblock {\em Publ. Math. Inst. Hungar. Acad. Sci.}, 5:17--61, 1960.

\bibitem{Frieze14}
A.~M. Frieze.
\newblock On a greedy 2-matching algorithm and {H}amilton cycles in random
  graphs with minimum degree at least three.
\newblock {\em Random Structures \& Algorithms}, 45(3):443--497, 2014.

\bibitem{Frieze19}
A.~M. Frieze.
\newblock Hamilton cycles in random graphs: a bibliography.
\newblock {\em ArXiv e-prints}, 1901.07139 [v13], July 2019.

\bibitem{FriezeKaronski}
A.~M. Frieze and M.~Karo\'nski.
\newblock {\em Introduction to Random Graphs}.
\newblock Cambridge University Press, Cambridge, UK, 2015.

\bibitem{FPPR}
A.~M. {Frieze}, X.~{P{\'e}rez-Gim{\'e}nez}, P.~{Pra{\l}at}, and B.~{Reiniger}.
\newblock {Perfect matchings and {H}amiltonian cycles in the preferential
  attachment model}.
\newblock {\em ArXiv e-prints}, October 2016.

\bibitem{FriezePittel}
A.~M. Frieze and B.~Pittel.
\newblock On a sparse random graph with minimum degree three: Likely
  {P}{\'{o}}sa sets are large.
\newblock {\em Journal of Combinatorics}, 4(2):123--156, 2013.

\bibitem{GaoIsaevMckay}
J.~Gao, M.~Isaev, and B.~Mc{K}ay.
\newblock Sandwiching random regular graphs between binomial random graphs.
\newblock {\em ArXiv e-prints}, 1906.02886, June 2019.

\bibitem{KomlosSzemeredi}
J.~Koml\'os and E.~Szemer\'edi.
\newblock Limit distribution for the existence of {H}amiltonian cycles in a
  random graph.
\newblock {\em Discrete Mathematics}, 43(1):55--63, 1983.

\bibitem{Korsunov}
A.D. Kor{\v{s}}unov.
\newblock Solution of a problem of {P}. {E}rd{\H{o}}s and {A}. {R}{\'e}nyi on
  {H}amiltonian cycles in a random graph.
\newblock {\em Dokl. Akad. Nauk SSSR}, 228:529--532, 1976.

\bibitem{KrivelevichLubetzkySudakov}
M.~Krivelevich, E.~Lubetzky, and B.~Sudakov.
\newblock Cores of random graphs are born {H}amiltonian.
\newblock {\em Proceedings of the London Mathematical Society}, 109:161--188,
  2014.

\bibitem{KrivelevichSudakovVuWormald}
M.~Krivelevich, B.~Sudakov, V.~Vu, and N.~Wormald.
\newblock Random regular graphs of high degree.
\newblock {\em Random Structures \& Algorithms}, 18(4):346--363, 2001.

\bibitem{RobinsonWormald}
R.~W. Robinson and N.~C. Wormald.
\newblock Almost all regular graphs are {H}amiltonian.
\newblock {\em Random Struct. Algorithms}, 5(2):363--374, April 1994.

\end{thebibliography}

\appendix

\section{On the existence of cutoffs}\label{app:cutoff}

We restate and prove Lemma~\ref{lem:cutoff}.
\begin{lemma*}[Lemma~\ref{lem:cutoff}]
Let $\m, \e > 0$. Suppose $\bp$ is a degree profile with minimum degree $\d$ and average degree at least $\d + \e$ and at most $\m$.
\begin{enumerate}[(i)]
\item There exists a cutoff for $\bp$.
\item If $\bp$ is linearly unbounded then there exists a cutoff $\CD$ with
$$
\mathrm{body}(\bp, \CD) + \mathrm{tail}(\bp,\CD) < \e.
$$
\end{enumerate}
\end{lemma*}
\begin{proof}
We prove (i). First let us show that there exists some constant $B$ such that for all $n$,
$$
\sum_{d = \d+1}^B dp_d^{(n)} \ge \frac\e2 n.
$$
Note that
$$
(\d + \e)n \le \sum_{d \ge \d} dp_d^{(n)} \le \d \left(p_\d^{(n)} + \sum_{d > \d} p_d^{(n)}\right) \le \d n + \d\sum_{d > \d}p_d^{(n)}
$$
so for all $n$,
$$
 \sum_{d > \d} p_d^{(n)}  \ge \frac\e\d n.
$$
Also, for any constant $B$,
$$
\sum_{d > B} p_d^{(n)} \le \frac1B \sum_{d > B} dp_d^{(n)} \le \frac{\m}{B} n.
$$
Choosing $B \ge 2\m \d / \e$ we then have
$$
\sum_{d = \d+1}^B dp_d^{(n)} \ge \d\left(\sum_{d > \d} p_d^{(n)} - \sum_{d > B}p_d^{(n)}\right) \ge \d\left(\frac\e\d - \frac\m B\right) n \ge \frac\e2n.
$$
Let $0 < \a < \e/8$ be some constant. Define
$$
D_n' = \inf\left\{ d > \d : dp_d^{(n)} \ge \a n\right\}.
$$
If $\sup D_n' < \infty$ then $(D_n', D_n')$ is a cutoff. Suppose $n$ is such that $D_n' = \infty$, so $dp_d^{(n)} < \a n$ for all $d > \d$. We have $(\d + \e)n \le \sum_d dp_d^{(n)}$ and $\d p_\d^{(n)} \le \d n$, so
\al{
  \e n & \le \sum_{d > \d} dp_d^{(n)}.
}
There must then exist some $\d < C_n'' \le D_n''$ such that
\al{
  \sum_{d = \d + 1}^{C_n''-1} dp_d^{(n)} & < \frac\e8 n \le \sum_{d = \d+1}^{C_n''}dp_d^{(n)} \le \frac\e8n + \a n < \frac\e4n, \\
  \sum_{d = C_n''}^{D_n'' - 1} dp_d^{(n)} & < \frac\e8 n \le \sum_{d = C_n''}^{D_n''} dp_d^{(n)} \le \frac\e8 n + \a n < \frac\e4 n.
}
Note that $\sum_{d = \d+1}^{D_n''} dp_d^{(n)} < \e n / 2$, so $D_n'' \le B$, and in particular $\sup D_n'' < \infty$. This shows that
$$
(C_n, D_n) = \left\{\begin{array}{ll}
(D_n', D_n'), & \text{ if $D_n' < \infty$}, \\
(C_n'', D_n''), & \text{ otherwise,}
\end{array}\right.
$$
is a cutoff for $\bp$.

We now prove (ii), so assume $\bp$ is linearly unbounded, i.e. for every constant $B$ there exists some $\beta > 0$ such that
$$
\liminf_n \frac1n \sum_{d\ge B} p_d^{(n)} \ge  \beta.
$$
Then any $\CD = (C_n, D_n)$ with $D_n = O(1)$ has $\mathrm{tail}(\bp, \CD) > 0$. Let $\e > 0$ be arbitrarily small. Since
$$
\frac1n\sum_{d\ge B}dp_d^{(n)} \le \m,
$$
there must exist some $C_n = O(1)$ and some $\beta > 0$ such that
$$
\beta \le \liminf_n \frac1n\sum_{d\ge C_n} dp_d^{(n)} < \frac\e2.
$$
This in turn implies that there must exist some $C_n \le D_n = O(1)$ for which
$$
\liminf_n \frac1n \sum_{d > D_n} dp_d^{(n)} < \b.
$$
This choice of $\CD = (C_n, D_n)$ has
$$
\mathrm{body}(\bp, \CD) + \mathrm{tail}(\bp, \CD) < \b + \frac\e2 < \e.
$$
\end{proof}

\end{document}